\documentclass{amsart}

\usepackage{amssymb, amsmath, amsthm, stmaryrd, amscd, mathrsfs}
\SetSymbolFont{stmry}{bold}{U}{stmry}{m}{n} 

\usepackage{verbatim}
\usepackage{a4wide}
\usepackage{dsfont}
\usepackage{enumerate}
\usepackage{hyperref}

\allowdisplaybreaks

\newtheorem{Def}{Definition}[section]
\newtheorem{Lem}[Def]{Lemma}
\newtheorem{Pro}[Def]{Proposition}

\newtheorem{The}[Def]{Theorem}

\numberwithin{equation}{section}

\newcommand{\dist}{\operatorname{dist}}
\newcommand{\N}{\mathcal N}
\newcommand{\R}{\mathcal R}
\newcommand{\spa}{\operatorname{span}}

\renewcommand{\hat}{\widehat}

\newcommand{\Div}{\operatorname{Div}}

\renewcommand{\phi}{\varphi}
\newcommand{\diag}{\operatorname{diag}}
\newcommand{\B}{\mathcal B}
\newcommand{\T}{\textsf T}

\newcommand{\ie}{i.e.\ }

\newcommand{\eg}{e.g.\ }

\newcommand{\Matrix}[1]{\left[\begin{matrix}#1 \end{matrix} \right]}
\newcommand{\scalar}[2]{\left(#1\middle| #2 \right)}
\newcommand{\ee}{\mathsf{e}}
\newcommand{\DD}{\mathds D}
\newcommand{\EE}{\mathds E}

\newcommand{\e}{\varepsilon}
\newcommand{\bnu}{\boldsymbol{\nu}}
\newcommand{\V}{\mathds V}

\usepackage{hyperref}

\def\typeout#1{\message{^^J}\message{#1}\message{^^J}}
\newif\ifSRCOK \SRCOKtrue
\newcount\PAGETOP
\newcount\LASTLINE
\global\PAGETOP=1
\global\LASTLINE=-1
\def\EJECT{\SRC\eject}
\def\WinEdt#1{\typeout{:#1}}
\gdef\MainFile{\jobname.tex}
\gdef\CurrentInput{\MainFile}
\newcount\INPSP
\global\INPSP=0
\def\SRC{\ifSRCOK%
  \ifnum\inputlineno>\LASTLINE%
    \ifnum\LASTLINE<0%
      \global\PAGETOP=\inputlineno%
    \fi%
    \global\LASTLINE=\inputlineno%
    \ifnum\INPSP=0%
      \ifnum\inputlineno>\PAGETOP%
        
      \fi%
    \else%
      
    \fi%
  \fi%
\fi}
\def\PUSH#1{%
\SRC%
\ifnum\INPSP=0 \global\let\INPSTACKA=\CurrentInput \else%
\ifnum\INPSP=1 \global\let\INPSTACKB=\CurrentInput \else%
\ifnum\INPSP=2 \global\let\INPSTACKC=\CurrentInput \else%
\ifnum\INPSP=3 \global\let\INPSTACKD=\CurrentInput \else%
\ifnum\INPSP=4 \global\let\INPSTACKE=\CurrentInput \else%
\ifnum\INPSP=5 \global\let\INPSTACKF=\CurrentInput \else%
               \global\let\INPSTACKX=\CurrentInput \fi\fi\fi\fi\fi\fi%
\gdef\CurrentInput{#1}%
\WinEdt{<+ \CurrentInput}%
\global\LASTLINE=0%
\ifSRCOK\fi%
\global\advance\INPSP by 1}
\def\POP{%
\ifnum\INPSP>0 \global\advance\INPSP by -1  \fi%
\ifnum\INPSP=0 \global\let\CurrentInput=\INPSTACKA \else%
\ifnum\INPSP=1 \global\let\CurrentInput=\INPSTACKB \else%
\ifnum\INPSP=2 \global\let\CurrentInput=\INPSTACKC \else%
\ifnum\INPSP=3 \global\let\CurrentInput=\INPSTACKD \else%
\ifnum\INPSP=4 \global\let\CurrentInput=\INPSTACKE \else%
\ifnum\INPSP=5 \global\let\CurrentInput=\INPSTACKF \else%
               \global\let\CurrentInput=\INPSTACKX \fi\fi\fi\fi\fi\fi%
\WinEdt{<-}%
\global\LASTLINE=\inputlineno%
\global\advance\LASTLINE by -1%
\SRC}
\def\INPUT#1{\relax}

\let\originalxxxeverypar\everypar
\newtoks\everypar
\originalxxxeverypar{\the\everypar\expandafter\SRC}
\everymath\expandafter{\the\everymath\expandafter\SRC}
\output\expandafter{\expandafter\SRCOKfalse\the\output}

\newif\ifSRCOK \SRCOKtrue
\newcount\PAGETOP
\newcount\LASTLINE
\global\PAGETOP=1
\global\LASTLINE=-1
\gdef\MainFile{\jobname.tex}
\gdef\CurrentInput{\MainFile}
\newcount\INPSP
\global\INPSP=0
\def\EJECT{\SRC\eject}
\def\WinEdt#1{\typeout{:#1}}
\def\SRC{\ifSRCOK%
  \ifnum\inputlineno>\LASTLINE%
    \ifnum\LASTLINE<0%
      \global\PAGETOP=\inputlineno%
    \fi%
    \global\LASTLINE=\inputlineno%
    \ifnum\INPSP=0%
      \ifnum\inputlineno>\PAGETOP%
      \fi%
    \else%
    \fi%
  \fi%
\fi}
\def\PUSH#1{%
\SRC%
\ifnum\INPSP=0 \global\let\INPSTACKA=\CurrentInput \else%
\ifnum\INPSP=1 \global\let\INPSTACKB=\CurrentInput \else%
\ifnum\INPSP=2 \global\let\INPSTACKC=\CurrentInput \else%
\ifnum\INPSP=3 \global\let\INPSTACKD=\CurrentInput \else%
\ifnum\INPSP=4 \global\let\INPSTACKE=\CurrentInput \else%
\ifnum\INPSP=5 \global\let\INPSTACKF=\CurrentInput \else%
               \global\let\INPSTACKX=\CurrentInput \fi\fi\fi\fi\fi\fi%
\gdef\CurrentInput{#1}%
\WinEdt{<+ \CurrentInput}%
\global\LASTLINE=0%
\ifSRCOK\fi%
\global\advance\INPSP by 1}
\def\POP{%
\ifnum\INPSP>0 \global\advance\INPSP by -1  \fi%
\ifnum\INPSP=0 \global\let\CurrentInput=\INPSTACKA \else%
\ifnum\INPSP=1 \global\let\CurrentInput=\INPSTACKB \else%
\ifnum\INPSP=2 \global\let\CurrentInput=\INPSTACKC \else%
\ifnum\INPSP=3 \global\let\CurrentInput=\INPSTACKD \else%

\ifnum\INPSP=4 \global\let\CurrentInput=\INPSTACKE \else%
\ifnum\INPSP=5 \global\let\CurrentInput=\INPSTACKF \else%
               \global\let\CurrentInput=\INPSTACKX \fi\fi\fi\fi\fi\fi%
\WinEdt{<-}%
\global\LASTLINE=\inputlineno%
\global\advance\LASTLINE by -1%
\SRC}
\def\INPUT#1{\relax}
\let\OldINCLUDE=\include
\def\include#1{
\EJECT%
\PUSH{#1.tex}%
\OldINCLUDE{#1}%
\POP}

\let\originalxxxeverypar\everypar
\newtoks\everypar
\originalxxxeverypar{\the\everypar\expandafter\SRC}
\everymath\expandafter{\the\everymath\expandafter\SRC}
\let\zzzxxxbibliography=\bibliography
\def\bibliography#1{\PUSH{\jobname.bbl}\zzzxxxbibliography{#1}\POP}
\output\expandafter{\expandafter\SRCOKfalse\the\output}

\begin{document}

\author{Martin Herberg}
\author{Martin Meyries}
\author{Jan Pr\"uss} 
\author{Mathias Wilke}
\email{martin.herberg@mathematik.uni-halle.de}
\email{martin.meyries@mathematik.uni-halle.de}
\email{jan.pruess@mathematik.uni-halle.de}
\email{mathias.wilke@mathematik.uni-halle.de}

\address{\!\!\!Martin-Luther-University Halle-Wittenberg, Institute of Mathematics, 06099 Halle (Saale), Germany}


\begin{abstract}
The mass-based Maxwell-Stefan approach to one-phase multicomponent reactive mixtures is mathematically analyzed. It is shown that the resulting quasilinear, strongly coupled reaction-diffusion system is locally well-posed in an $L_p$-setting and generates a local semiflow on its natural state space. Solutions regularize instantly and become strictly positive if their initial components are all nonnegative and nontrivial. For a class of reversible mass-action kinetics, the positive equilibria are identified: these are precisely the constant chemical equilibria of the system, which may form a manifold. Here the total free energy of the system is employed which serves as a Lyapunov function for the system. By the generalized principle of linearized stability, positive  equilibria are proved to be normally stable. 
\end{abstract}

\title[Maxwell-Stefan Diffusion]{Reaction-diffusion systems of Maxwell-Stefan type with reversible mass-action kinetics}


\maketitle

\section{Introduction}
\subsection{Reaction-diffusion systems of Maxwell-Stefan type} The Maxwell-Stefan approach modeling diffusion in multicomponent mixtures is well-known in the engineering literature, cf. \cite{Gio, KT86, KW97, RTao, WK00}. 
In the mathematical community the resulting reaction-diffusion equations seem much less known, but have recently attracted a lot of attention, see \cite{Bothe, BGS12, JS13}.  Therefore, we begin with a review of the basic ideas of this approach. 

Let $\Omega\subset \mathbb R^n$ be an open bounded domain with boundary $\partial\Omega$ of class $C^{2+\alpha}$ and outer normal field $\nu$. We consider a mixture of $N\geq 2$ species $A_k$  with \emph{molar masses} $M_k>0$ and \emph{individual mass densities} $\rho_k\geq0$ filling the container $\Omega$. Mass balance of the single component $A_k$ reads
$$ \partial_t\rho_k + {\rm div}_x(\rho_k \mathbf {u}_k) = M_k r_k\qquad \mbox{ in  } \Omega,\quad t>0,$$
where $\mathbf {u}_k$ denotes the \emph{individual velocity} of species $A_k$, satisfying $(\mathbf {u}_k|\nu)=0$ on $\partial\Omega$, and $r_k$ is the rate of production of species $A_k$ due to chemical reactions. Observe that the kinetics $r$ should be {\em positivity preserving}, i.e., subject to the condition
$$ \rho_j\geq0,\quad \rho_k=0 \quad \Rightarrow \quad r_k\geq0,$$
and should satisfy $\sum_k M_kr_k =0$, which results in \emph{conservation of total mass}. The quantities of interest are the mass densities $\rho_k$, while the individual velocities $\mathbf {u}_k$ are in general unknown and have to be modeled, as well as the kinetics $r_k$. To reduce the complexity of these balance laws, we introduce the \emph{total density} $\rho=\sum_k \rho_k$, the {\em barycentric velocity} $\mathbf {u}=\sum_k\rho_k\mathbf {u}_k/\rho$, the \emph{mass fractions} $y_k=\rho_k/\rho$, and the \emph{concentrations} $c_k=\rho_k/M_k= y_k \rho/M_k$. With these new variables, we obtain the overall mass  balance
$$
 \partial_t \rho + {\rm div}_x(\rho \mathbf {u})=0\qquad  \mbox{ in } \Omega,\quad t>0,
$$
and $(\mathbf {u}| \nu)=0$ on $\partial\Omega$.
The individual mass balances now become
$$ \rho(\partial_t y_k +\mathbf {u}\cdot \nabla_x y_k) + {\rm div}_x J_k = M_k r_k\qquad \mbox{ in } \Omega,\quad t >0,$$
where the \emph{diffusive fluxes} $J_k$ are given by
$$ J_k = \rho_k(\mathbf {u}_k-\mathbf {u}),\qquad k=1,\ldots,N.$$
Note that, by definition, $\sum_k y_k=1$ and $\sum_k J_k=0$.

So far everything is physically exact in the framework of continuum mechanics.
However, to obtain a closed model one has to prescribe laws for $\mathbf {u}$, $r_k$, and most importantly for the diffusive fluxes $J_k$.
In this paper, we are interested in the {\em incompressible, isobaric, isothermal case}, which means  
$$\rho = const, \qquad \mathbf u = 0,$$
and no temperature dependence. We note that most of the engineering literature, as well as the papers \cite{Bothe, JS13}, is \emph{molar-based}, i.e., instead of the total mass $\rho$ the \emph{total molar concentration} $c_{\text{tot}} = \sum_k c_k$ is assumed be constant and the \emph{molar averaged velocity} $\mathbf v = \frac{1}{c_{\text{tot}}}\sum_k c_k \mathbf u_k$ vanishes. However, adding up the individual mass balances, this leads to $\sum_k r_k = 0$, which is only satisfied in special situations. Hence, also having in mind the more general case of nontrivial velocity field and temperature, we prefer the  \emph{mass-based ansatz} $\rho = const$.

The above assumptions lead to the problem
\begin{equation}\label{MSDR}
 \rho\partial_ty_k +{\rm div}_x J_k =M_kr_k(y)\quad \mbox{ in } \Omega,\qquad (J_k|\nu)=0\quad \mbox{ on } \partial\Omega,
\end{equation}
for $k=1,\dots,N$, completed by initial data $y_k(0) = y_0^k \geq 0$.
We again emphasize  the constraints
\begin{equation}\label{consmass}
\sum_{k=1}^N J_k=0,\qquad \sum_{k=1}^N y_k=1.
\end{equation}
Together  with $y\geq0$ this already implies $L_\infty$-bounds for $y$, a very important property. Therefore, when modeling the diffusive fluxes it is essential that positivity
as well as conservation of mass are ensured.

A classical approach to model the diffusive fluxes $J_k$ is now as follows. One of the species, say $A_N$, acts as a solvent for the mixture, say water, or tuluol, benzol, etc. This means that $y_N$ is close to 1 and
the remaining $y_k$ are small, hence the $A_k$ are dilute for $k\neq N$. As $A_N$ will in general not be involved  in the reactions, the equation for $y_N$ is ignored and the remaining diffusive fluxes are modeled by \emph{Fick's law}, i.e.\
$J_k =-d_k\nabla_xy_k$, $k=1,\ldots,N-1$, where $d_k>0$. This way \eqref{MSDR} becomes a semilinear reaction-diffusion system with diagonal main part which preserves nonnegativity of $y_1,\ldots,y_{N-1}$.
However, now \eqref{consmass} forces an unrealistic diffusive flux $J_N$ for $y_N$, such that nonnegativity and \emph{a priori} $L_\infty$-bounds for the mass fractions might get lost. This might be one reason for the notorious problem of global existence in the Fickian approach, see \cite{Pierre, Rolland}. Another drawback of this approach is that cross-diffusion effects like uphill diffusion or osmotic diffusion cannot be modeled, but are well-known to appear in nature, cf. \cite{Duncan}.

An alternative way to model the diffusive fluxes is the \emph{Maxwell-Stefan approach}, which goes back to the old but famous papers \cite{Maxwell, Stefan}.
In this approach, a balance of so-called {\em driving forces} ${\sf d}_k$ and {\em friction forces} ${\sf f}_k$ is postulated, i.e., ${\sf d}_k={\sf f}_k$.
The friction forces are modeled by 
\begin{equation}\label{friction-forces}
{\sf f}_k = \rho\sum_{j\neq k} f_{kj}y_ky_j(\mathbf u_j-\mathbf u_k) = \sum_{j\neq k} f_{kj}(y_kJ_j-y_jJ_k),
\end{equation}
see \cite[Formula (16)]{KW97}, with the symmetric \emph{friction coefficients} $f_{kj}=f_{jk}>0$.  These coefficients may depend on the composition $y$, but in the sequel we assume them to be constant. Observe that $\sum_k \textsf{f}_k = 0$, so that the friction forces act only on the components but not on the mixture.
The driving forces ${\sf d}_k$ have to be modeled as well and are typically given by the \emph{chemical potentials} $\mu_k$. 

In the mentioned literature, where the total molar concentration $c_{\text{tot}}$ is constant instead of the total mass as in the present paper, it is assumed that ${\sf d}_k = c_k \nabla_x \mu_k$, and one usually considers the standard chemical potentials $\mu_k = \log(\gamma_k c_k)$ for the potentials. Here $R$ is the universal gas constant, $\theta$ is the (constant) temperature and $\gamma_k > 0$ are so-called activity coefficients. This results in ${\sf d}_k =  \nabla_x c_k$. In particular, the necessary relation $\sum_k {\sf d}_k = 0$ is satisfied due to $\sum_k c_k  =const$.  

However, in the mass-based approach and for general chemical potentials $\mu_k = \partial_{y_k} \psi$, 
where $\psi$ is the density of the constitutive \emph{Helmholtz free energy}, $\sum_k {\sf d}_k = 0$ is in general not satisfied anymore and thus the ansatz for $\textsf{d}_k$ must be modified. Based on entropy considerations, it will be demonstrated in \cite{BotheDreyer} that
\begin{equation}\label{driving-forces}
{\sf d}_k =  y_k( \nabla_x\mu_k-\sum_{j=1}^N y_j\nabla_x\mu_j)
\end{equation}
is in fact a natural modification of the above driving force model which guarantees $\sum_k {\sf d}_k = 0$ for arbitrary free energies $\psi$. In the sequel we assume \eqref{driving-forces} and 
\begin{equation}\label{eq:889}
 \rho \psi =  \sum_{k} c_k[ \log(c_k/\mathbf{c}_*^k) - 1] =  \sum_{k} \frac{\rho y_k}{M_k} [ \log(y_k/\mathbf{y}_*^k) - 1],
\end{equation}
which results in $\mu_k = \frac{1}{M_k} \log(y_k/\mathbf{y}_k^*)$ for the potentials. 
 Here $\mathbf{c}_*^k = \frac{\rho}{M_k} \mathbf{y}_*^k$ are the components of a constant \emph{chemical equilibrium} $\mathbf{c}_*$ (see below and Section \ref{sec:stability}).

The assumptions ${\sf d}_k={\sf f}_k$ lead to the \emph{Maxwell-Stefan equations} for the columns $J^\alpha \in \mathbb R^N$ of the flux matrix $J = (J_1,\ldots, J_N)^{\textsf{T}} \in \mathbb R^{N\times n}$. Writing
$$M={\rm diag}(M_j),\qquad \textsf{e}=[1,\ldots,1]^{\sf T},\qquad  P(y)=I- y\otimes \textsf{e} = I-(\cdot|\textsf{e})y,$$
these equations read as follows:
\begin{equation}\label{MSLaw}\left \{
\begin{array}{l}
B(y) J^\alpha = P(y)M^{-1}\partial_{x_\alpha}y,\quad \alpha=1,\ldots,n,\\\\
B(y)=[b_{ij}(y)], \quad b_{ij}(y)= f_{ij}y_i\; \mbox{ for } i\neq j,\quad b_{ii}(y)=-\sum_{l=1}^N f_{il}y_l, \quad i,j=1,\ldots,N.
\end{array}\right.
\end{equation}
Recall that $f_{ij} = f_{ji} > 0$ are constants, and we set $f_{ii} = 0$. As a consequence of the mass-based approach explained above, \eqref{MSLaw} differs in particular from the equations considered in \cite{Bothe, JS13} by the projection $P(y)$ onto 
$$\mathds E = \{\textsf{e}\}^{\perp}.$$
At this point one essentially has to solve in \eqref{MSLaw} for the $J^\alpha$ and insert the result into \eqref{MSDR}, leading to a system of reaction-diffusion equations for the mass fractions $y_k$ satisfying $\sum_k y_k = 1$.

\subsection{Main results} We now describe the results of the present paper concerning the Maxwell-Stefan equations \eqref{MSLaw} as well as  solvability, positivity and stability of equilibria for \eqref{MSDR}.

It was demonstrated in \cite{Bothe}, employing the Perron-Frobenius theory for quasi-positive matrices, that $B(y)$ is invertible on $\mathds E$ for all $y$ from 
$$\mathring \DD =\{y\in (0,1)^N: (y|\textsf{e}) = 1\}.$$
In Section 2 we extend the analysis of $B(y)$ and show that it is in fact invertible on a (relatively) open neighbourhood $\mathds V\subset \mathds E + \frac{1}{N} \textsf{e}$ of
$$\DD =\{y\in [0,1]^N: (y|\textsf{e}) = 1\},$$
i.e., we may allow for $y$ with vanishing and even negative components. Observing that $(J_k|\nu)=0$ is equivalent to $\partial_\nu y=0$ on $\partial\Omega$, from \eqref{MSDR} we arrive at the following quasilinear, strongly coupled parabolic system
\begin{equation}
\label{MS}\tag{MS} \left \{
\begin{split}
\rho\partial_t y + {\rm Div}_x ( A(y)P(y)M^{-1}[\nabla_x y]^{\sf T})&= Mr(y) &&\mbox{ in } \Omega,\;\;t>0,\\
\partial_\nu y&=0 &&\mbox{ on } \partial\Omega,\;\;t>0,\\
y(0)&=y_0&&\mbox{ in } \Omega,
\end{split}\right.
\end{equation}
where $A(y) = (B(y)|_{\mathds E})^{-1}$. Here we write $\nabla_x y = [\partial_\alpha y_j] \in \mathbb R^{n\times N}$ and $\text{Div}_x$ means to take the divergence in each row of the $N\times n$-matrix $A(y)P(y)M^{-1}[\nabla_x y]^{\sf T}$.

It turns out that for all $y\in \mathds V$ the negative flux matrix $-A(y)P(y)M^{-1}$ is \emph{normally elliptic} in $\mathds E$, i.e., its spectrum satisfies $\sigma(-A(y)P(y)M^{-1}|_{\mathds E}) \subset \{\text{Re}\,z > 0\}$. As  a consequence, the linearization of \eqref{MS} enjoys the property of \emph{maximal $L_p$-regularity}, $1<p<\infty$. This allows to prove \emph{local-in-time existence and uniqueness of classical solutions} for \eqref{MS} for sufficiently smooth initial data $y_0$ with values in $\mathds V$ (Theorems \ref{Sat:Haupt1} and \ref{thm:classical}). Here we may allow for general mass preserving kinetics $r$. The result is based on the general theory from \cite{KohPruWil, Bari} for quasilinear parabolic problems, which we summarize in the appendix for the reader's convenience.

For initial data with values in $\mathring \DD$, local well-posedness by means of maximal $L_p$-regularity was already indicated in \cite{Bothe}. Our extension to a neighbourhood $\mathds V$ has the following advantages. First, we may allow for initial data with components vanishing on parts or even on all of $\Omega$, which is desirable from an applications point of view. Secondly, it allows to show that  solutions corresponding to initial data with  nontrivial components become \emph{instantaneously strictly positive} (Theorem \ref{positivity}). This result will be achieved as follows. In Lemma \ref{Lem:Inverse} we demonstrate that the inverse $A(y) = [a_{ij}(y)]=  (B(y)|_{\mathds E})^{-1}$ can be represented by coefficients
$$ a_{ij}(y) = y_i a_{i}^1(y),\quad  j\neq i,\qquad a_{ii}=-a_i^0(y),$$
where $a_i^0$ and $a_{ij}^1$ are real analytic and $a_i^0(y)>0$ for $y_i = 0$. As a consequence, it turns out that a component $y_i$ is a supersolution  of a linear parabolic equation (see Theorem \ref{eq:669}), provided $y_i$ is close to zero. Since we know that $y_i$ is smooth including the points where it (hypothetically) vanishes, we may apply the strong maximum principle and Hopf's lemma to deduce strict positivity. Thie result is valid for  mass and positivity preserving kinetics $r$.

In our investigations of stability of equilibria in Section 5 we specialize to {\em mass-action kinetics}, modeling $m\in \mathbb N$ single reversible  reactions of the species $A_j$,
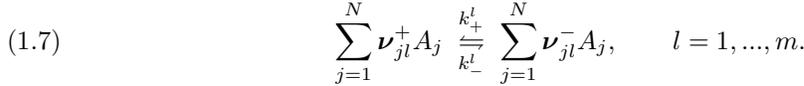
\begin{equation}\label{eq:elem-reac}
\sum_{j=1}^N \bnu_{jl}^+ A_j\; \overset{k_+^l}{\underset{k_-^l}{\leftrightharpoons}} \; \sum_{j=1}^N \bnu_{jl}^- A_j, \qquad l=1,...,m.
\end{equation}
Here $\bnu_{jl}^+, \bnu_{jl}^- \in \mathbb N_0$ are the \emph{stoechiometric coefficients} and $k_+^l, k_-^l > 0$ are the \emph{reaction rates}.  The precise form of the corresponding kinetics $r(y)$ is described in Section \ref{sec:stability}, see also (\cite[Section 6.4]{Gio}). Such mass-action kinetics are always positivity preserving, and we assume them to be mass conserving. More importantly, we assume that at least one \emph{chemical equilibrium} $\mathbf c_*$ exists, i.e., each single reaction in \eqref{eq:elem-reac} is at equilibrium in $\mathbf c_*$. Then any kinetic equilibrium with stricitly positive components is a chemical one. Further, the set  of all strictly positive equilibria of \eqref{MS} forms a smooth manifold whose dimension equals $N-s-1$, where $s< N$ is the rank of the \emph{stoechiometric matrix} $\bnu = [\bnu_{jl}^+ - \bnu_{jl}^-] \in \mathbb Z^{N\times m}$.

Under the above assumptions, we are going to show that any positive equilibrium is a homogeneous kinetic one (Proposition \ref{Lem:Konstanten}), and that each of these are stable, as $t\to \infty$, with respect to the semiflow generated by \eqref{MS} (Theorem \ref{Sat:Haupt2}). Further, each solution starting sufficiently close to the set of equilibria is global-in-time and converges exponentially fast to a single equilibrium. This generalizes \cite[Theorem 9.7.4]{Gio} to the case of familiy a equilbria, i.e., when the rank of $\bnu$ is less than $N-1$.

The key to these results is the \emph{total free energy}
$$
\Psi(y) = \int_\Omega \psi(y)\,dx,
$$
with density $\psi$ defined in \eqref{eq:889}, which serves as a Lyapunov function for \eqref{MS}. In fact, along smooth positive solutions $y$ we have
$$\rho\partial_t\psi(y)+{\rm div}_x \big(\sum_k \mu_k J_k\big) = \sum_k (\nabla_x\mu_k|J_k) + \sum_k \mu_kM_kr_k,$$
and it will be shown in Section \ref{sec:stability} that
$$\sum_k (\nabla_x\mu_k|J_k)\leq0, \quad  \sum_k \mu_k M_kr_k\leq 0.$$
Moreover, if both these quantities vanish, then $J =0$ and we are in a spatially homogeneous chemical equilibrium. This remarkable property of the Maxwell-Stefan model characterizes the equilibria of \eqref{MS}. Our stability proof is based on the \emph{generalized principle of linearized stability} for manifolds of equilibria in quasilinear problems \cite{PruSimZac}.  Here the essential point is to determine the kernel of the linearization and to show that zero is a semi-simple eigenvalue.

We finally give a conditional result on the convergence to equilibria of globally bounded solutions which stay away from the boundary $\DD\setminus \mathring\DD$ (see Proposition \ref{prop:global}) as $t\to \infty$. This will be a consequence of the relative compactness of bounded orbits and the Lyapunov property of the free energy. Compactness follows from the method of time weights \cite{KohPruWil, PruSim}. Already in the ODE case, the analysis of solutions that converge to $\DD\setminus \mathring\DD$ is rather difficult, see \cite{HJ72}.

We expect that our approach can be extended to the case of variable total density $\rho$, when combined with a Navier-Stokes equation for the barycentric velocity $\mathbf u$. This topic will be addressed in another paper.

Let us mention other analytical results on reaction-diffusion systems based on the Maxwell-Stefan approach. In the mathematically pioneering article \cite{Bothe}, the author already sketches some ideas which are used in our paper. These include normal ellipticity of the linearization as well as an argument to prove nonnegativity. In \cite[Theorem 9.7.4]{Gio}, for reversible mass-action kinetics as above global existence of classical solutions and their convergence as $t\to \infty$ is shown in a neighbourhood of an isolated positive equilibrium. In fact, in \cite{Gio} the case $\Omega = \mathbb R^n$ and a constant, nontrivial velocity field $\mathbf u$ is considered. The arguments are based on energy methods. In \cite{BGS12} the case $N=3$ is investigated in a special situation involving equality of some friction coefficients. \emph{Global existence} of nonnegative weak solutions for general positive
initial data is proven in \cite{JS13} by considering \eqref{MS} in entropy variables. The result is proved under the \emph{a priori} assumption that the solution of \eqref{MS} is strictly positive for all times. Uniqueness of such solutions is not known. For vanishing kinetics $r=0$ it is further shown that the constructed solution converges to the mean value of the initial data. In the compressible case, global weak solutions are constructed in \cite{MPZ12} for a two-component mixture. \smallskip

This paper is organized as follows. In Section \ref{sec:MS} we study the Maxwell-Stefan equations \eqref{MSLaw} in detail. In Section \ref{sec:wrp} we prove local-in-time well-posedness, regularity and instantaneous positivity for (MS). Section \ref{sec:stability} contains the stability analysis of equilibria. In the appendix we summarize the abstract results for general quasilinear parabolic problems that are used in the paper.\smallskip

\textbf{Notations.} The space of linear operators between Banach spaces $X_1,X_0$ is denoted by $\mathcal B(X_1, X_0)$, and $\mathcal B(X_0) = \mathcal B(X_0,X_0)$. Kernel, range and spectrum of an operator $A$ are denoted by $\N(A)$, $\R(A)$ and $\sigma(A)$, respectively. For a vector $y\in \mathbb R^N$ we write $y\geq 0$ resp. $y > 0$ if $y_k\geq 0$ resp. $y_k > 0$ for each $k=1,\ldots,N$. We further write $Y = \diag (y_k)\in \mathbb R^{N\times N}$ for $y\in \mathbb R^N$. Throughout we will consider the following subsets of $\mathbb R^N$, where $\textsf{e} = (1,...,1)^{\textsf{T}}\in \mathbb R^N$,
$$\EE = \{\textsf{e}\}^\perp, \qquad  \mathring{\DD} = \{y\in (0,1)^N\,:\, (\textsf{e}|y) = 1\}, \qquad \DD = \{y\in [0,1]^N\,:\, (\textsf{e}|y) = 1\}.$$

\section{Inversion of the Maxwell-Stefan relations}\label{sec:MS}

In this section we investigate the Maxwell-Stefan relations \eqref{MSLaw} in more detail. We show that the restriction of the matrix $B(y)$ to $\EE$ is invertible for all $y$ in an open neighbourhood $U\subset \mathbb R^N$ of $\DD$, and investigate the structure of its inverse
$$A(y) = (B(y)|_{\EE})^{-1}.$$
We further show that the spectrum of the negative flux matrix in \eqref{MS},
$$A_0(y) = -A(y) P(y) M^{-1},$$
considered as an element of $\mathcal B(\EE)$, belongs to $(0,\infty)$ for all $y \in U$.

The following properties of $B(y)$ were obtained in \cite[Section 5]{Bothe}, as a consequence of the Perron-Frobenius theorem for irreducible, quasi-positive matrices. See also \cite[Lemma 3]{JS13}.

\begin{Lem}[]\label{Lem:B}
For any $y\in \mathbb R^N$ we have $y\in \N(B(y))$ and $\R(B(y)) \subseteq \EE$.
Moreover, for $ y\in \mathring{\DD}$ it holds that
$$\sigma(B(y)) \subseteq (-\infty,0], \qquad N(B(y)) = \spa \{ y\},\qquad R(B(y)) = \{\ee\}^\bot=\EE.$$
\end{Lem}

This lemma shows in particular that $B(y)$ may be restricted to an
element $B(y)|_{\EE}$ of $\mathcal B(\EE)$ for all
$y\in \mathbb R^N$. We show that $B(y)|_{\EE}$ is invertible for $y$ from a larger set containing $\DD$.

 \begin{Lem}\label{Lem:Inverse}
There is an open neighbourhood $U\subset \mathbb R^N$ of $\mathds
D$ such that for all $y\in U$ the restriction $B(y)|_{\EE}$
of $B(y)$ to $\EE$ is invertible. Denote its inverse by
$A(y) = (B(y)|_{\EE})^{-1}$. Then there are real analytic
functions $a_{i}^0, a_{ij}^1\colon U\to\mathbb R$ such that for all
$y\in U$ and $h\in \EE$ the vector $x
= A(y)h$  may be represented by
$$x_i = -a_i^0(y)h_i + y_i \sum_{j^\neq i} a_{ij}^1(y)h_j, \qquad i=1,...,N.$$
Moreover, we have $a_i^0(y) > 0$ for $y_i = 0$.
\end{Lem}
\begin{proof} \emph{Step 1.} We show that $B(y)|_{\EE}$ is invertible for $y\in \DD$. For
$y\in \mathring{\DD}$ this follows already from Lemma
\ref{Lem:B}. So let $y\in \DD$ be such that $y_{k} =0$ for
some  $1\leq k\leq N$. Assume $B(y)x = 0$ for $x\in \EE$.
We show  $x=0$. The structure of $B(y)$ from \eqref{MSLaw}
implies that $b_{kj}  = 0$ for $j\neq k$ and $b_{kk} =
-\sum_{l=1}^N f_{kl} y_l$. Thus $b_{kk}x_k = 0$. Because of
$f_{kl} > 0$ and $(\textsf{e}|y)=1$ we have $b_{kk} \neq 0$, and
therefore $x_k = 0$. In this way $B(y)x = 0$ reduces to $\hat
B(\hat y) \hat x = 0$, where the $(N-1)\times (N-1)$-matrix $\hat
B(\hat y)$ results from deleting the $k$-th row and the $k$-th
column of $B(y)$, and $\hat \xi = (\xi_1,\ldots, \xi_{k-1},
\xi_{k+1},\ldots, \xi_N)^{\textsf{T}}$ for $\xi\in \mathbb R^N$.
Since $(\hat {\textsf{e}}|\hat y)=1$, the matrix $\hat B(\hat y)$
has the same structure as $B(y)$ in \eqref{MSLaw}. Hence, if
other components $y_{k_2},\ldots, y_{k_m}$ of $y$ vanish, we may
argue as before to obtain $x_{k_2}, ..., x_{k_m} = 0$. In case
$m=N-1$ we immediately obtain $x=0$ since $x\in \EE$. If $m
< N-1$, the remaining components $\tilde x$ of $x$ satisfy $\tilde
B (\tilde y) \tilde x = 0$, where $\tilde B (\tilde y)$ is again
as in \eqref{MSLaw}, $(\tilde{\textsf{e}}|\tilde y) = 1$ and the
components of $\tilde y$ do not vanish. Since $(\tilde
{\textsf{e}}|\tilde x) = 0$, Lemma \ref{Lem:B} applies to $\tilde
B (\tilde y)$ and shows that $\tilde x = 0$. Altogether, it
follows that $x=0$, hence $B(y)|_{\EE}$ is injective. As
$\EE$ is finite dimensional we obtain the invertibility of
$B(y)|_{\EE}$ for all $y\in \DD$. Since $B(y)$
depends continuously on $y$, we obtain an open neighbourhood $U$
of $\DD$ such that $B(y)|_{\EE}$ is invertible for
all $y\in U$.

\emph{Step 2.} To investigate the structure of $A(y) = (B(y)|_{\EE})^{-1}$ for $y\in U$ we introduce the matrix
$$D(y)= \Matrix{B(y) & y \\ \ee^{\textsf T} & 0}.$$
We claim that $D(y)$ is invertible on $\mathbb R^{N+1}$. Indeed,
for given $h\in \mathbb R^N$ and $\beta \in \mathbb R$ the
solution $\Matrix{x \\\alpha} \in \mathbb R^{N+1}$ of $D(y)
\Matrix{x
\\\alpha} = \Matrix{h \\\beta}$ is
$$x = (B(y)|_{\EE})^{-1}(h- (\textsf{e}|h)y) + \beta y, \qquad \alpha = (\textsf{e}|h).$$
Now fix $h\in \EE$. With $\beta = 0$, this yields a
representation of $x = A(y)h$ in terms of $D(y)^{-1}$, \ie,
$\Matrix{x \\ 0} = D(y)^{-1}\Matrix{h \\0}$. Let $D_i(y)$ be the
matrix that results from replacing the $i$-th column of $D(y)$ by
$\Matrix{h \\0}$. Then $x_i = \frac{\det D_i(y)}{\det D(y) }$ for
$i=1,\ldots,N$ by Cramer's rule. Developing $ D_i(y)$ with respect to
the $i$-th column, we obtain $\det D_i(y) = \sum_{j=1}^N
(-1)^{i+j}h_j \det \hat D^{ji}(y)$, where $\hat D^{ji}(y)$ is the
matrix that results from deleting the $j$-th row and the $i$-th
column of $D(y)$. Now assume $j\neq i$. By \eqref{MSLaw}, a row
of $\hat D^{ji}(y)$ is given by $y_i(f_{i1},\ldots, f_{i,i-1},
f_{i,i+1},..., f_{i,N-1}, 1)$. Developing $\hat D^{ji}(y)$ with
respect to this row, we obtain that $\det \hat D^{ji}(y)$ is a
multiple of $y_i$. This yields the representation
$$x_i = -a_{i}^0(y) h_i + y_i \sum_{j\neq i}a_{ij}^1(y)h_j,$$
with coefficients analytic in $y\in U$. It remains to prove that $a_{i}^0(y) > 0$ for $y_i = 0$. In this case the structure of $B(y)$ yields $b_{ii} x_i = h_i$, where $b_{ii} = - \sum_{j=1}^N f_{ij} y_j < 0$ for $y$ sufficiently close to $\DD$. Hence $a_{i}^0(y) = -1/b_{ii} > 0$. We have thus shown that $x = (B(y)|_{\EE})^{-1}h$ may be represented as asserted.
\end{proof}

We next investigate the spectrum on $\EE$ of the negative flux matrix $A_0(y)=-A(y)P(y)M^{-1}$ in \eqref{MS}. To this end
we employ a well-known symmetrization of $B(y)$ for $y \in \mathring{\DD}$. Define
$$y^{1/2} = (y_1^{1/2}, \ldots, y_n^{1/2} )^{\textsf{T}},\qquad Y^{1/2} = \diag(y^{1/2}).$$
Then we have 
\[ B_S(y) := Y^{-1/2} B(y) Y^{1/2} =\Matrix{-s_1 & & \bar{d}_{ij}\\ &\ddots & \\  \bar{d}_{ij} & &-s_n}, \]
where $s_i= \sum_{k=1}^N f_{ik}y_k$ and $  \bar{d}_{ij} =
f_{ij}(y_iy_j)^{1/2}$. Observe that $B_S(y)$ is symmetric and
$\sigma(B_S(y)) \subset (-\infty,0]$ by Lemma \ref{Lem:B}. Its
kernel and range $B_S(y)$ are given by $\N(B_S(y)) =
\spa\{y^{1/2}\}$ and $\R(B_S(y)) = \{y^{1/2}\}^\bot$, such that
$B_S(y)$
 is invertible on $\{y^{1/2}\}^\bot$.

\begin{Lem} \label{lem:specA_0}Consider $A_0(y)$ as an element of $\mathcal B(\EE)$. Then there is an open neighbourhood $U\subset \mathbb R^N$ of $\DD$ such that for all $y\in U$ the spectrum of $A_0(y)$ belongs to $\{\text{\emph{Re}}\,z > 0\}$.
\end{Lem}
\begin{proof} As $A_0$ depends continuously on $y$, it suffices to show that $\sigma_{\EE} (A_0(y)) \subset (0,\infty)$ for $y\in \DD$, since then we obtain $\sigma_{\EE} (A_0(y)) \subset \{\text{Re}\,z > 0\}$  for all $y$ from a sufficiently small neighbourhood $U$ of $\DD$. Throughout, let $\lambda$ be an eigenvalue of $A_0(y)$ with eigenvector $v\in \EE$, such that $P(y)M^{-1} v= - \lambda B(y)v$.

\emph{Step 1.} Assume $y \in \mathring {\DD}$. Using that $Y^{-1} = Y^{-1} P(y) + (\cdot|\textsf{e})\textsf{e}$ and $(v|\textsf{e}) = 0$, we get
\[0<\scalar{v}{Y^{-1}M^{-1} v}=\scalar{v}{Y^{-1}P(y)M^{-1} v} = -\lambda\! \scalar{v}{ Y^{-1} B(y) v} =-\lambda \!\scalar{w}{ B_S(y) w},\]
where $w=Y^{-1/2}v$. Since $B_S(y)$ is negative semidefinite, we
obtain $\lambda > 0$.

\emph{Step 2.} Assume $y \in \DD$ is such that $y_k = 0$ for some $1\leq k \leq N$.  We write
\begin{equation}\label{666}
 -\lambda B(y) v = P(y)M^{-1}v = M^{-1}v - (M^{-1}v|\textsf{e})y.
\end{equation}
By the structure of $B(y)$ from \eqref{MSLaw}, here the $k$-th equation reads $-\lambda b_{kk}(y) v_k = M_k^{-1} v_k$,
where $b_{kk}(y) <0$. Hence we either have $\lambda > 0$ and are finished, or $v_k = 0$. In the latter case, the equation
\eqref{666} reduces to
$$-\lambda \hat B(\hat y) \hat v = \hat M^{-1}\hat v - (\hat M^{-1}\hat v|\hat{\textsf{e}})\hat y = \hat P(\hat y)\hat M^{-1}\hat v,$$
where the hat means to delete the $k$-th row and the $k$-th
column for a matrix and to delete the $k$-th entry for a vector.
If $y$ has no further vanishing components we are in the
situation of Step 1 and conclude $\lambda > 0$. Otherwise, if
$y_{k_2}, \ldots,y_{k_m} = 0$, we obtain inductively that either
$\lambda > 0$ or $v_{k_2}, \ldots, v_{k_m} = 0$, where necessarily
$m<N-1$. In the latter case, as above we can reduce to the
situation of Step 1, and $\lambda
> 0$ follows.
\end{proof}

For later purposes we investigate $-A(y)P(y)Y$ in more detail.

\begin{Lem} \label{Lem:4} For $y\in \mathring{\DD}$ the matrix $-A(y)P(y)Y$ is symmetric and positive semi-definite.
 The restriction $-A(y)P(y)Y|_{\EE}$ is positive definite.
\end{Lem}
\begin{proof}
To show the symmetry we let $P_{y^{1/2}}= I -
(\cdot|y^{1/2})y^{1/2} $ be the orthogonal projection  onto
$\{y^{1/2}\}^{\bot}$. Observing that $A(y) = Y^{1/2}
(B_S(y)|_{\EE})^{-1} Y^{-1/2}$, $P(y)Y = Y P(y)^\textsf{T}$
and $P_{y^{1/2}} = Y^{-1/2} P(y) Y^{1/2}$, and recalling that the
range of $(B_S(y)|_{\EE})^{-1}$ equals
$\{y^{1/2}\}^{\bot}$, for $v,w\in \mathbb R^{N}$ we calculate
\begin{align*}
\left (A(y)P(y)Y v | w \right )
&= \left (Y^{1/2}(B_S(y)|_{\EE})^{-1} Y^{1/2} P(y)^{\textsf{T}}  v \middle| w \right ) \\
&= \left (P_{y^{1/2}}(B_S(y)|_{\EE})^{-1} Y^{1/2} P(y)^{\textsf{T}}  v \middle| Y^{1/2} w \right )\\
&= \left (v \middle|P(y) Y^{1/2} (B_S(y)|_{\EE})^{-1} P_{y^{1/2}} Y^{1/2} w \right )\\
&= \left (v \middle| A(y) P(y)Y  w \right ).
\end{align*}
The inclusion $\sigma(-A(y)P(y)Y|_{\EE}) \subseteq (0,\infty)$
follows as in Step 1 of the proof of Lemma \ref{lem:specA_0},
replacing $M^{-1}$ by $Y$. Hence $-A(y)P(y)Y|_{\EE}$ is
positive definite. Since $\mathbb R^N =
\text{span}\{\textsf{e}\}\oplus \EE$ and $\textsf{e} \in
\N(-A(y)P(y)Y)$, we obtain that $-A(y)P(y)Y$ is positive
semi-definite.
\end{proof}

\section{Well-posedness, regularity and positivity}\label{sec:wrp}

\subsection{Well-posedness} We apply the general results from \cite{KohPruWil, Bari}, which are summarized in the appendix, to obtain local-in-time well-posedness for \eqref{MS}. Let us first reformulate \eqref{MS} in the abstract form \eqref{eq:Abstrakt}, i.e., 
$$\dot u + \mathcal{A}(u)u=F(u),\quad  t >0,\qquad u(0)=u_0.$$
Define the spaces
$$X_0 =L_p(\Omega;\mathds{E}), \qquad X_1 =\{ u\in W_p^2(\Omega; \mathds E) \mid \partial_\nu u=0 \}.$$
In the sequel we will assume that $p>n+2$, wherefore the embedding $W_p^{2-2/p}(\Omega;\mathds{E})\hookrightarrow C^1(\overline{\Omega};\mathds{E})$ is at our disposal, see \cite[Theorem 4.6.1]{TriInterpolation}. Here the spaces $W_p^s$ for $s\notin\mathbb{N}$ denote the Sobolev-Slobodeckij spaces, see \cite[Section 4.2.1]{TriInterpolation}. In this case one also has
$$W_p^{2\mu-2/p}(\Omega;\mathds{E})\hookrightarrow C^1(\overline{\Omega};\mathds{E}),$$
provided that $\mu>\mu_0:=(n+2)/2p+1/2$. Note that for $u\in W_p^{2\mu-2/p}(\Omega;\mathds{E})$ with $\mu\in (\mu_0,1]$, the Neumann trace $\partial_\nu u$ on $\partial\Omega$ exists. Therefore the trace space $X_{\gamma,\mu}=(X_0,X_1)_{\mu-1/p,p}$ is given by
$$X_{\gamma,\mu}=\{u\in W_p^{2\mu-2/p}(\Omega;\mathds{E})\,:\,\partial_\nu u=0\},$$
see \cite[Theorem 4.3.3]{TriInterpolation}. Let $U\subset \mathbb R^N$ be the open neighborhood of $\mathds{D}$ from Lemma \ref{Lem:Inverse}, 
$$\V= U \cap (\ee/N + \EE)$$
a relative open set in $\ee/N+\EE$ containing $\DD$,
and define
$$V_\mu=\{u\in X_{\gamma,\mu}\,:\,u(\overline{\Omega})+{\ee}/N\in \V\}.$$
Then $V_\mu$ is an open subset of $X_{\gamma,\mu}$, since $X_{\gamma,\mu}\hookrightarrow C(\overline{\Omega};\mathds{E})$. For all $u\in V_\mu$ and all $v\in X_1$ we define the substitution operators $\mathcal{A}\colon V_\mu\to\mathcal{B}(X_1,X_0)$ and $F\colon V_\mu\to X_0$ by
\begin{align}
\mathcal{A}(u)v(x)&= -\Div (A_0(u(x)+{\ee}/N) [\nabla v(x)]^{\T}) \notag\\
&=  - A_0(u(x)+{\ee}/N)\Delta v(x)-\sum\limits_{j=1}^n \left[ \sum\limits_{l=1}^N \partial_{l} A_0(u(x)+{\ee}/N) \partial_{j}u_l(x)  \right] \partial_{j}v(x),\quad x\in\Omega\label{eq:opmathcalA},
\end{align}
and
$$F(u)(x)=Mr(u(x)+ \textsf{e}/N),\qquad x\in\Omega.$$
In order to apply Theorem \ref{The:lokaleExistenz} we have to show that for each $u\in V_\mu$ the operator $\mathcal{A}(u)$ has maximal regularity of type $L_p$ and that
$$(\mathcal{A},F)\in C^1(V_\mu;\mathcal{B}(X_1,X_0)\times X_0).$$
\begin{Lem}\label{Pro:Diffbarkeit}
Let $p>n+2$, $\mu\in (\mu_0,1]$ and assume that $r\in C^1(U;\mathbb R^N)$ satisfies $(Mr|\emph{\textsf{e}}) = 0$. Then $\mathcal A\in C^1(V_\mu; \B(X_1,X_0))$, $F\in C^1(V_\mu;X_0)$ and the derivative of $\mathcal{A}$ is given by
$$
[\mathcal A'(u) h] v = - [A_0'(u+{\ee}/N)]h\Delta v
-\sum\limits_{j=1}^n \left[ \sum\limits_{l=1}^N \partial_{j}h_l \partial_{l} A_0(u+{\ee}/N) + \partial_{j}u_l [\partial_{l} A_0'(u+{\ee}/N)]h  \right] \partial_{j}v,
$$
where $u\in V_\mu$, $v\in X_1$ and $h\in X_{\gamma,\mu}$.
\end{Lem}

\begin{proof}
We know from Lemma \ref{Lem:Inverse} that the mapping $[U\ni y\mapsto A_0(y)\in \mathcal{B}(\mathds{E})]$ is real analytic, in particular it is $C^1$. It follows readily that the mapping
$$[u\mapsto\mathcal{A}(u)],\quad \{u\in C^1(\overline{\Omega};\mathds{E}):u(\overline{\Omega})+\textsf{e}/N\subset \V\}\to \mathcal{B}(X_1;X_0),$$
is continuously Fr\'{e}chet differentiable. This in turn implies that $\mathcal{A}\in C^1(V_\mu;\mathcal{B}(X_1;X_0))$, since by assumption the embedding $X_{\gamma,\mu}\hookrightarrow C^1(\overline{\Omega};\mathds{E})$ is valid.
Applying the same strategy to $F$ yields $F\in C^1(V_\mu;X_0)$. 
\end{proof}

We will now show that for each $u\in V_\mu$ the operator $\mathcal{A}(u)$ has maximal regularity of type $L_p$. By Lemma \ref{lem:specA_0}, the principal part $\mathcal{A}_{\#}(u(x))=-A_0(u(x)+ \frac{1}{N}\textsf{e})\Delta$ of $\mathcal A(u(x)+ \frac{1}{N}\textsf{e})$ is \emph{normally elliptic} for each $u\in V_\mu$ and $x\in \overline{\Omega}$, i.e., $\sigma (-A_0(u(x)+ \frac{1}{N}\textsf{e})) \subset \{\text{Re}\,z > 0\}$. Furthermore, for each $u\in V_\mu$, the Neumann boundary operator $\partial_\nu$ satisfies the \emph{Lopatinskii-Shapiro condition} with respect to $\mathcal{A}_{\#}(u(x))$ (normal complementing condition, see \eg \cite[Section 8]{DenHiePru}). To be precise, it holds that for each $x\in\partial\Omega$, all $\lambda \in \overline{\mathbb{C}_+}$ and all $\xi \in \mathbb R^{n-1}$ with $|\lambda| + |\xi| \neq 0$ the only decaying solution $v\in C(\mathbb{R}_+;\mathds{E})$ of the ODE system
$$
\lambda v(\tau) - A_0(u(x)) ( - |\xi|^2 v(\tau) +  v''(\tau) )=0,\quad \tau>0, \qquad v'(0)=0,
$$
is $v=0$. This follows from the spectral properties of  $A_0(u(x)+\frac{1}{N}\textsf{e})$.

Therefore,  \cite[Theorem 8.2]{DenHiePru} yields that for each $u\in V_\mu$, the operator $\mathcal{A}(u)$, defined in \eqref{eq:opmathcalA}, has maximal regularity of type $L_p$. We are now in a position to apply Theorem \ref{The:lokaleExistenz} which yields the following well-posedness result for \eqref{MS}. 
\begin{The}\label{Sat:Haupt1}
Let $\Omega\subset\mathbb{R}^n$ be a bounded domain with boundary $\partial\Omega\in C^2$, let  $p>n+2$ and $\mu\in (\mu_0,1]$. Suppose that $r\in C^1(U;\mathbb{R}^N)$ and $(Mr(y)|\emph{\textsf{e}})=0$ for all $y\in \V$. Then the following assertions are valid.
\begin{itemize}
\item[a)] For each $y_0\in W_p^{2\mu-2/p}(\Omega;\mathbb R^N)$ with $y_0(\overline{\Omega})\subset \V$ and $\partial_\nu y_0=0$ at $\partial\Omega$, there exists $T>0$ and a unique solution
\begin{equation}\label{eq:regy}
y\in W_{p,\mu}^1(0,T;L_p(\Omega;\mathbb R^N))\cap L_{p,\mu}(0,T;W_p^2(\Omega;\mathbb R^N))\cap BUC(0,T;W_p^{2\mu-2/p}(\Omega;\mathbb R^N))
\end{equation}
of \eqref{MS} with $y(t,x)\in\V$ for all $(t,x)\in [0,T]\times\overline{\Omega}$.
\item[b)] Each local solution can be extended to a maximal solution defined on a maximal interval of existence $J(y_0)=[0,t^+(y_0))$ and \eqref{eq:regy} holds for each $T\in (0,t^+(y_0))$. The mapping $y_0\mapsto t^+(y_0)$ is lower semicontinuous and the mapping $y_0\mapsto y(\cdot,y_0)$ is continuously Fr\'{e}chet differentiable.
\item[c)] For each  $T\in (0,t^+(y_0))$ we have
$$y\in C^1((0,T];W_p^{2-2/p}(\Omega;\mathbb R^N))\cap C^{2-1/p}((0,T];L_p(\Omega;\mathbb R^N))\cap C^{1-1/p}((0,T];W_p^2(\Omega;\mathbb R^N)).$$
\end{itemize}
\end{The}

\subsection{Classical solutions}
In the situation of the above theorem, let us show that the  solution $y$ of \eqref{MS} is in fact classical, i.e.,
$$y\in C^1((0,T];C(\overline{\Omega};\mathbb R^N))\cap C((0,T];C^2(\overline{\Omega};\mathbb R^N))$$
for each $T\in (0,t^+(y_0))$ if $\partial\Omega \in C^{2+\alpha}$ for some $\alpha > 0$. Theorem \ref{Sat:Haupt1} already yields that
$y\in C^1((0,T];C(\overline{\Omega};\mathbb R^N))$, since $W_p^{2\mu-2/p}(\Omega)$ is embedded into $C(\overline{\Omega})$ whenever $p>n+2$ and $\mu\in (\mu_0,1]$.

Therefore it remains to show that $y\in C((0,T];C^2(\overline{\Omega};\mathbb R^N))$. To this end, we write the equation for $y$ in terms of $u=y-{\ee}/N$ as $-A(t,x)\Delta u(t,x)=g(t,x)$, where $A(t,x)=A_0(u(t,x)+{\ee}/N)$ and
$$g(t,x)=\sum\limits_{j=1}^N \left[ \sum\limits_{l=1}^n \partial_{l} A_0(u(t,x)+{\ee}/N) \partial_{j}u_l(t,x)  \right] \partial_{j}u(t,x)-\partial_t u(t,x)+Mr(u(t,x)+{\ee}/N).$$
By Theorem \ref{Sat:Haupt1} and Sobolev's embedding, there exists $\alpha\in (0,1)$ such that $A\in C^\alpha((0,T]\times\overline{\Omega};\mathcal{B}(\mathds{E}))$ and $g\in C^\alpha((0,T]\times\overline{\Omega};\mathds{E})$. Note that for fixed $t_*\in (0,T)$ the matrix $A(t_*,x)$ is invertible for each $x\in\overline{\Omega}$. This yields the equation
$-\Delta u(t_*,x)=A(t_*,x)^{-1}g(t_*,x)$, complemented by the boundary condition $\partial_\nu u(t_*,x)=0$ for $x\in\partial\Omega$. From now on we assume that $\partial\Omega\in C^{2+\alpha}$. Then it follows from \cite[Theorem 6.31]{GilTru} that $u(t_*,\cdot)\in C^{2+\alpha}(\overline{\Omega};\EE)$ and that there exists a constant $C>0$, which does not depend on $t_*\in (0,T)$, such that the estimate
$$\|u(t_*,\cdot)\|_{C^{2+\alpha}(\overline{\Omega};\EE)}\le C\left(\|A(t_*,\cdot)^{-1}g(t_*,\cdot)\|_{C^{\alpha}(\overline{\Omega};\EE)}+\|u(t_*,\cdot)\|_{C^{\alpha}(\overline{\Omega};\EE)}\right)$$
is valid. Hence $u\in C((0,T);C^{2+\alpha}(\overline{\Omega},\mathds E))$ and we have proven the following result.
\begin{The}\label{thm:classical}
Let the conditions of Theorem \ref{Sat:Haupt1} be satisfied and assume that $\partial\Omega\in C^{2+\alpha}$ for some $\alpha >0$. Then the unique solution of \eqref{MS} is a classical solution.
\end{The}

\subsection{Positivity} Assuming the kinetic term $r$ to be positivity preserving, we show the nonnega\-tivity of solutions of \eqref{MS}, and the instantaneous strict positivity of components corresponding to nontrivial initial data. The argument heavily relies on the
structure of  the diffusion term $\text{Div}_x(A_0(y)[\nabla_x y]^\textsf{T})$.

We consider this structure in more detail. Since $A_0(y) = -A(y)P(y)M^{-1}$ with $A(y) =
(B(y)|_{\EE})^{-1}$ from Lemma \ref{Lem:Inverse} and $P(y) =
I - (\cdot|\textsf{e})y$, the $i$-th component of $\text{Div}_x(A_0(y)[\nabla_x y]^\textsf{T})$ is given by
$$-\sum_{\alpha=1}^n \sum_{j=1}^N  a_{ij}(y) (M_j^{-1} \partial_{x_\alpha}^2 y_j - \partial_{x_\alpha}
[(M^{-1}\partial_{x_\alpha} y|\textsf{e})y_j]) -
\partial_{x_\alpha}(a_{ij}(y)) [ M_j^{-1} \partial_{x_\alpha} y_j
- (M^{-1}\partial_{x_\alpha} y|\textsf{e})y_j],$$ where $a_{ii}(y)
= -a_i^0(y)$ and $a_{ij}(y) = y_i a_{ij}^1(y)$ for $j\neq i$. We
collect the summands with $j=i$ from the first term, which
results in $-M_i^{-1}a_{i}^0(y)\Delta y_i$. All the other
summands contain either $\partial_{x_\alpha} y_i$ or $y_i$ as a
factor. Thus, as long as it exists, a component $y_i$ of a
solution of \eqref{MS} satisfies an equation of the form
\begin{equation}\label{eq:669}
\rho\partial_t y_i - M_i^{-1}a_i^0(y)\Delta_x y_i +
\sum_{\alpha=1}^n b_{i\alpha}^0(t,x)\partial_{x_\alpha} y_i +
c_i^0(t,x) y_i = M_ir_i(y),
\end{equation}
with coefficients $b_{i\alpha}^0, c_i^0$ depending on the partial derivatives up to
second order of $y$. We further write the $i$-th reaction
term $M_i r_i$ as
\begin{equation}\label{eq:671}
 M_i r_i(y) = -y_iL_i + h_i(y),
\end{equation}
where $L_i >0$ is the Lipschitz constant of $M_i r_i$ on $\DD$ and, with $\hat y = (y_1,..., y_{i-1}, 0, y_{i+1}, ... y_N)$,
$$h_i(y) = M_i r_i(\hat y) + L_iy_i + M_i(r_i(y) - r_i(\hat y)) \geq 0, \qquad y\in \DD.$$
Here $r_i(\hat y) \geq 0$ follows from the assumption that $r$ is positivity preserving.

Combining \eqref{eq:669} and \eqref{eq:671}, we arrive at
\begin{equation}\label{eq:700}
\rho\partial_t y_i - M_i^{-1}a_i^0(y)\Delta_x y_i +
\sum_{\alpha=1}^n b_{i\alpha}^0(t,x)\partial_{x_\alpha} y_i +
(c_i^0(t,x)+L_i) y_i \geq 0,
\end{equation}
Since $a_i^0(y) > 0$ for $y_i = 0$, the left-hand side of \eqref{eq:700} is \emph{parabolic} for $y_i$ close to zero, and the lower order coefficients $b_{i\alpha}^0, c_i^0$ are continuous if $y$ is a classical solution. This puts us into a position to apply \emph{maximum principles} and \emph{Hopf's lemma}, which is the key to the following result on nonnegativity and strict positivity. Recall that $U\subset \mathbb R^N$ denotes the neighborhood of $\mathds{D}$ from Lemma \ref{Lem:Inverse}.

\begin{The} \label{positivity} Assume $r\in C^1(U,\mathbb R^N)$ is mass and positivity preserving on $\DD$. Let for $p > n+2$ the initial data $y_0\in W_p^{2-2/p}(\Omega; \mathbb R^N)$ with $y_0(\overline{\Omega})\subset \V$ and $\partial_\nu y_0 = 0$ be given. Denote by $y$ the corresponding unique classical solution of \eqref{MS}. Then the following holds true.
\begin{itemize}
 \item[\emph{a)}] If $y_0 \geq 0$, then $y(t,x)\geq 0$ for all $t\in (0,t^+(y_0))$ and $x\in \overline{\Omega}$
 \item[\emph{b)}] If $y_0 \geq 0$ and $y_0^i \neq 0$, then $y_i(t,x) > 0$ for all $t\in (0,t^+(y_0))$ and $x\in \overline{\Omega}$.
\end{itemize}
\end{The}

\begin{proof} \emph{Step 1.} We prove Part a) (see also \cite[Section 6]{Bothe}). Let $y_0\geq 0$. For $\varepsilon > 0$ we consider the modified system
\begin{equation}\label{eq:667}
 \rho\partial_t y^\e + \text{Div}_x (A_0(y^\e) [\nabla_x y^\e]^{\textsf{T}}) = M r^\e, \qquad \partial_\nu y^\e =0, \qquad y^\e(0) = y_0^\e,
\end{equation}
with reaction terms $Mr^\e = Mr + \varepsilon (\textsf{e} -Ny^\e)$ and initial data  $y_0^\e = y_0 + \varepsilon(\textsf{e}-Ny_0)$. Observe that $(\textsf{e}|Mr^\e) = 0$ for $(\textsf{e}|y^\e)=1$, that $(\textsf{e}|y_0^\e) = 1$ and that $y_0^\e$ has strictly positive components for all sufficiently small $\e$. Thus \eqref{eq:667} has a unique maximal classical solution $y^\e$ by Theorem \ref{Sat:Haupt1}.

Fixing $\e> 0$, we claim that $y^\e(t,x) > 0$ for all $t\in
[0,t^+(y_0^\e))$ and $x\in \overline{\Omega}$. Assume the
contrary, i.e., there are $t_0\in (0,t^+(y_0^\e))$ and $x_0 \in
\overline{\Omega}$ such that $y_i^\e(t_0,x_0) = 0$ for a
component $y_i$ and $y_j^\e(t,\cdot)> 0$ on $\overline{\Omega}$
for all $j=1,...,N$ and $t\in [0,t_0)$. Note that necessarily
$t_0> 0$ since $y_0^\e > 0$. First suppose that $x_0\in \Omega$.
Then $\partial_t y_i^\e(t_0,x_0) \leq 0$, $\nabla_x
y_i^\e(t_0,x_0) = 0$ and $\Delta_x y_i^\e(t_0,x_0) \geq 0$.
Further, $(Mr^\e)_i =M_i r_i + \e(1-Ny_i^\e) \geq \e$
 at $(t_0,x_0)$ since $r_i \geq 0$ for $y_i=0$.  Therefore  \eqref{eq:669}
 yields
$$\rho\partial_t y_i^\e(t_0,x_0) - M_i^{-1} a_i^0(y^\e(t_0,x_0)) \Delta_x y_i^\e(t_0,x_0)  \geq \e,$$
a contradiction. Suppose next that $x_0\in \partial \Omega$. Then
$\partial_\nu y_i(t_0,x_0) \leq 0$. On the other hand,
\eqref{eq:700}  implies that there is $\eta > 0$ such that
$y_i^\e$ is a supersolution of a linear parabolic equation in
$(t_0-\eta,t_0]\times V$, where $V\subset \Omega$ is a
sufficiently small open ball with $x_0\in \partial V$. The
previous considerations show that $y_i^\e> 0$ in
$(t_0-\eta,t_0]\times V$. Hence $\partial_\nu y_i(t_0,x_0) > 0$
by Hopf's lemma, see Theorem 3.7 (and the remark thereafter) of
\cite{Protter}, which again leads to a contradiction. We conclude
that $y^\e > 0$ on $(0,t^+(y_0^\e))\times \overline{\Omega}$.

Given $T \in (0,t^+(y_0))$, we obtain $y^\e\to y$ as $\e\to 0$ in
$C([0,T];W_p^{2-2/p}(\Omega,\mathbb R^N))$ from Theorem \ref{The:lokaleExistenz}, and thus uniformly
on $[0,T] \times \overline{\Omega}$. Hence $y\geq 0$ on
$[0,t^+(y_0))\times \overline{\Omega}$.

\emph{Step 2.} We prove Part b) and assume additionally that  $y_0^i\neq 0$. From Step 1 we know  $y_i\geq 0$. For $t\in (0,t^+(y_0))$ we set
$$\Omega_t^+ = \{x\in \Omega\,:\, y_i(t,x) > 0\}.$$
We are going to show that $\Omega_t^+$ is nonempty, open and closed in $\Omega$. Clearly, $\Omega_t^+$ is open in $\Omega$. To obtain $\Omega_t^+ \neq \emptyset$, let $t_0$ be the smallest time such that $\Omega_{t_0}^+ =\emptyset$, \ie, $y_i(t_0,\cdot) = 0$ on $\Omega$. Note that $t_0 > 0$ by the assumption $y_0^i \neq 0$. Then the left-hand side of  \eqref{eq:700} is parabolic in $(t_0-\eta,t_0]\times \Omega$ for small $\eta$.  Since $y_i$ attains its minimum zero everywhere on $\{t_0\}\times \Omega$, the strong maximum principle yields $y_i(t_0-\eta,\cdot) = 0$, see again \cite[Theorem 3.7]{Protter}. But this is a contradiction to the definition of $t_0$, and therefore $\Omega_t^+\neq \emptyset$ for all $t$.

We finally show that $\Omega_t^+$ is closed in $\Omega$. Let $x_k\in \Omega_t^+$ be a sequence such that $x_k\to x_0\in \Omega$ as $k\to \infty$. Assume $x_0\notin \Omega_t^+$, \ie, $y_i(t,x_0) = 0$. Then there are $\eta > 0$ and a convex open set $V\subset \Omega$ containing $x_0$ such that \eqref{eq:700} is parabolic on $(t-\eta,t]\times V$. As above, by the strong maximum priniciple, $y_i(t,\cdot) = 0$ on $V$. Hence $y_i(t,x_k) = 0$ for all sufficiently large $k$, which contradicts the assumption $x_k\in \Omega_t^+$.

We conclude that $\Omega_t^+ = \Omega$ for each $t$, and therefore $y_i > 0$ on $(0,t^+(y_0))\times \Omega$. Arguing as in the previous step by contradiction and Hopf's lemma, we get  $y_i > 0$ on $(0,t^+(y_0))\times \overline{\Omega}$.
\end{proof}


\section{Stability of equilibria and long-time behavior} \label{sec:stability}
For a class of reversible mass-action kinetics $r$ modeling \eqref{eq:elem-reac} we show that \eqref{MS} only has spatially homogeneous kinetic equilibria in $\mathring \DD$, that any of these equilibria is stable and that solutions starting sufficiently close to an equilibrium converge exponentially as $t\to \infty$. 
\subsection{Reversible mass-action kinetics} 
The reversible mass-action kinetic $r$ modeling \eqref{eq:elem-reac} is given by (see also \cite[Section 6.4]{Gio})
\begin{equation}\label{eq:888}
 r(y) = \bnu \mathbf{r}(\rho M^{-1}y) = \sum_{l=1}^m \bnu_l \mathbf r_l(\rho M^{-1}y),
\end{equation}
where $\bnu = [\bnu_{jl}^+-\bnu_{jl}^-] \in \mathbb Z^{N\times m}$ is for $\bnu_{jl}^+, \bnu_{jl}^-\in \mathbb N_0$ the stoechiometric matrix of \eqref{eq:elem-reac} and $\bnu_l\in \mathbb Z^N$ denotes the $l$-th column of $\bnu$. The vector $\mathbf r = (\mathbf r_1, \ldots, \mathbf r_m)$ of elementary reactions is in terms of the concentrations $c$ given by
$$\mathbf{r}_l(c) = -k_l^+ c^{\bnu_l^+} + k_l^-c^{\bnu_l^-}, \qquad l=1,...,m,$$
where $\bnu_l^+ = [\bnu_{jl}^+]$ and $\bnu_l^- = [\bnu_{jl}^-]$, such that $\bnu_l = \bnu_l^+ - \bnu_l^-$ for a column of $\bnu$. Here we use multiindex notation, i.e., $c^{\bnu_l^+} = c_1^{\bnu_{1l}^+}\cdot\ldots \cdot c_N^{\bnu_{Nl}^+}$, and analogous for $c^{\bnu_l^-}$. The set of positive chemical kinetic equilibria, i.e., where all elementary reactions $\mathbf r_l$ vanish, is given by $$\mathcal E = \{y_*\in \mathring {\DD}\,:\, \mathbf r(\rho M^{-1} y_*) = 0\}.$$
The \emph{stoechiometric subspace} $\mathbb S$ of $\mathbb R^N$ is defined by
$$\mathbb S = \R(\bnu), \qquad s = \dim \mathbb S.$$
We assume that the columns $\bnu_l$ of $\bnu$ are ordered such that $\bnu_1,\ldots,\bnu_s$ are linearly independent, i.e.,  
$$\mathbb S = \text{span}\{\bnu_1,\ldots,\bnu_s\}.$$
Throughout we make the following assumptions: 
\begin{equation}
 \mathcal E \neq \emptyset, \qquad M\textsf{e}\in \mathbb S^{\perp},  \qquad k_l^+, k_l^- > 0.\tag{\textbf{R}}
\end{equation}
Observe that the second condition implies $(\textsf{e}|Mr(y)) = 0$ for each $y$, i.e., conservation of mass. It also implies that $s < N$. The strict positivity of $k_l^+, k_l^-$ means that each elementary reaction $\mathbf{r}_l$ is reversible. It is straight forward to check that mass-action kinetics $r$ as above are positivity preserving, see \cite[Lemma 6.4.3]{Gio}. Concerning the first condition, the existence of a chemical equilibrium can be characterized as follows. Observe that $y_*\in \mathcal E$ if and only if for $c_* = \rho M^{-1} y_*$ we have
\begin{equation} \label{eq:chem-eq}
c_*^{\bnu_l} = \frac{c_*^{\bnu_l^+}}{c_*^{\bnu_l^-}} =
\frac{k_l^-}{k_l^+} =: K_l, \qquad l=1,\ldots,m.
\end{equation}
Since $k_l^+, k_l^- > 0$, \eqref{eq:chem-eq} is equivalent to 
\begin{equation}\label{eq:chem-eq-eq}
 (\bnu_l|\log c_*) = \log K_l, \qquad l=1,...,m,
\end{equation}
where we write $\log \xi = (\log \xi_1, \ldots, \log \xi_N)^\textsf{T}$ for a vector $\xi\in \mathbb R^N$. By the linear independence of the $\bnu_l$, for $l=1,..., s$ the equations in \eqref{eq:chem-eq-eq} can always be satisfied. For the remaining equations we note that there are $\alpha_{lk}\in \mathbb R$ such that $\bnu_l = \sum_{k=1}^s \alpha_{lk} \bnu_k$ for $l=s+1,\ldots,m.$ Thus $\mathcal E\neq\emptyset $ if and only if 
$$\prod_{k=1}^s K_k^{\alpha_{lk}} = K_l,\qquad l=s+1,...,m.$$
In particular,  $\mathcal E\neq\emptyset $ is always satisfied in case $s > m$. On the other hand, if $\mathcal E \neq \emptyset$, then the equations in \eqref{eq:chem-eq-eq} for $l=s+1,...,m$ become redundant. 

In the sequel we fix an arbitrary positive chemical equilibrium $\mathbf y_* \in \mathcal E$ and consider the chemical potentials
$$\mu_k(y) = \frac{1}{M_k} \log (y_k/\mathbf y_*^k), \qquad y\in \mathring {\DD}, \qquad k=1,\ldots,N.$$
We will also write 
$$\mu= (\mu_1,\ldots, \mu_N)^{\textsf{T}} = M^{-1} \log(y/\mathbf y_*).$$

\begin{Lem}\label{lem:manifold} Assume \emph{(\textbf{R})}. Then the following holds true. 
\begin{itemize}
 \item[a)] We have $(\mu(y)|Mr(y))\leq 0$ for all $y\in \mathring \DD$, and equality holds if and only if $y\in \mathcal E$.
  \item[b)] For $y\in \mathring \DD$ we have $r(y) = 0$ if and only if $y\in \mathcal E$, i.e., all kinetic equilbria are chemical.
   \item[c)] The set $\mathcal E$ forms an $(N-s-1)$-dimensional smooth submanifold of $\mathbb R^N$. At $y_*\in \mathcal E$, the tangent space is given by $T_{y_*} \mathcal E = \N(\bnu^{\mathsf T}Y_*^{-1})\cap \EE$.
\end{itemize}
\end{Lem} 
\begin{proof} \emph{Step 1.} Writing $c = \rho M^{-1} y$ for $y\in \mathring \DD$,  and $\mathbf c_* = \rho M^{-1} \mathbf y_*$, we have $\mu(y) = M^{-1} \log(c/\mathbf c_*)$. Using \eqref{eq:chem-eq}, we calculate
\begin{align}
(\mu(y)|Mr(y)) & =  \sum\limits_{l=1}^m \left( \log(c/\mathbf c_*) \middle| \bnu_l \right) \left(-k_l^+ c^{\bnu_l^+} + k_l^-  c^{\bnu_l^-} \right) \notag \\
&= -\sum\limits_{l=1}^m  \log[(c/\mathbf c_*)^{\bnu_l}] k_l^+
c^{\bnu_l^-}\mathbf c_*^{\bnu_l} \left( \frac{c^{\bnu_l}}{\mathbf c^{\bnu_l}_*} - 1 \label{eq:m691}
\right).
\end{align}
Since $k_l^+c^{\bnu_l^-}\mathbf c_*^{\bnu}> 0$ and $(\log \xi)(\xi-1)\geq 0$ for all $\xi >0$, we obtain  $(\mu(y)| M r(y)) \leq 0$. Since each summand in \eqref{eq:m691} is nonpositive, we have $(\mu(y)| M r(y)) = 0$ if and only $c^{\bnu_l}  =\mathbf c_*^{\bnu_l} = K_l$ for each $l$, also using \eqref{eq:chem-eq}.  This implies $y\in \mathcal E$ and proves a). Assertion b) is a direct consequence of a).

\emph{Step 2.} We prove c). Above we have seen that that the assumption $\mathcal E\neq \emptyset$ implies that $y_*\in \mathcal E$ if and only if $(\bnu_l|\log c_*) = \log K_l$ for $l=1,\ldots s$. Hence $\mathbf r(\rho M^{-1} y_*) = 0$ is equivalent to $\mathbf r_1(\rho M^{-1}y_*),\ldots, \mathbf r_s(\rho M^{-1} y_*) = 0$.

Define the map
$$F(y) = [ \mathbf r_1(\rho M^{-1} y),\ldots,  \mathbf r_s(\rho M^{-1} y), (y|\textsf{e}) - 1]^{\textsf{T}},\qquad y\in \mathring \DD,$$
such  that $F(y_*) = 0$ if and only if $y_*\in \mathcal E$. We show that $F'(y_*)$ has full rank $s+1$ at each $y_*\in \mathcal E$. To this end we calculate, writing  $c_* = \rho M^{-1} y_*$ and using \eqref{eq:chem-eq},
$$\partial_{y_j} \mathbf r_l(\rho M^{-1} y_*) =\rho M_j^{-1}(-k_l^+ \bnu_{jl}^+c_*^{\bnu_l^+} + k_l^-
\bnu_{jl}^-c_*^{\bnu_l^-})/c_*^j= -k_l^- c_*^{\bnu_{l}^-} \bnu_{jl} /y_*^j.$$
Therefore, writing $\hat \bnu = ( \bnu_1,\ldots, \bnu_s)\in \mathbb Z^{N\times s}$,
$$F'(y_*) = \Matrix{- \text{diag}(k_l^-c_*^{\bnu_{l}^-}) \hat {\bnu}^{\textsf{T}} Y_*^{-1} \\ \textsf{e}^{\textsf{T}}}.$$ 
Since $k_l^-c_*^{\bnu_{l}^-}> 0$ and $\hat{\bnu}^{\textsf{T}}$ is surjective, the matrix $-\text{diag}(k_l^-c_*^{\bnu_{l}^-}) \hat {\bnu}^{\textsf{T}} Y_*^{-1} $ has full rank $s$. We further claim that $\textsf{e}$ is linearly independent of the rows  of $-\text{diag}(k_l^-c_*^{\bnu_{l}^-}) \hat {\bnu}^{\textsf{T}} Y_*^{-1} $. If this were not the case, we find numbers $\lambda_l$ such that $\textsf{e} = \sum_{l=1}^s \lambda_l Y_*^{-1} \bnu_l$. Taking the scalar product with $MY_* \textsf{e}$, we obtain from $M\textsf{e} \in \mathbb S^{\perp}$ that 
$$0\neq( MY_* \textsf{e}|\textsf{e}) = \sum_{l=1}^s \lambda_l  (\bnu_l|M\textsf{e}) = 0,$$
a contradiction. Hence $F'(y_*)$ has full rank $s+1$, which implies that $\mathcal E$ is a smooth manifold of dimension $N-s-1$. The tangent space is given by 
$$\N(F'(y_*)) = \N(\hat {\bnu}^{\textsf{T}}Y_*^{-1})\cap \EE = \N(\bnu^{\textsf{T}}Y_*^{-1})\cap \EE.$$
\end{proof}

\subsection{Stability of equilibria} For reversible mass-action kinetics as in \eqref{eq:888}, in this subsection we show that all equilibria of \eqref{MS} are stable spatially homogeneous kinetic equilibria. Throughout we fix $p > n+2$ and set
$$\mathcal X = \{y_0 \in W_p^{2-2/p}(\Omega,\mathbb R^N)\;:\; y_0(\overline{\Omega})\subset \mathring \DD, \;\;\partial_{\nu} y_0 = 0\}.$$
Then Theorem \ref{Sat:Haupt1} yields that \eqref{MS} is locally well-posed on $\mathcal X$. The key quantity to identify the equlibria is the total free energy $\Psi$, which is given by
$$\Psi(y) = \int_\Omega \psi(y)\,dx,\qquad \psi(y) = \sum_{k=1}^N \frac{y_k}{M_k}(\log(y_k/ \mathbf y_*^k) - 1).$$
Observe that $\mu_k = \partial_{y_k} \psi$ for the chemical potentials.

\begin{Pro}\label{Lem:Konstanten} Assume \emph{(\textbf{R})}. Then the following holds true.
\begin{itemize}
\item[\text{a)}] The total free energy $\Psi$ is a strict Lyapunov function for \eqref{MS} on 
$\mathcal X$.
\item[\text{b)}] Each equilibrium of \eqref{MS} is spatially homogeneous, such that the set of equilibria of \eqref{MS} in $\mathcal X$ is given by (the constant functions in) $\mathcal E$.
\end{itemize}
\end{Pro}
\begin{proof} \emph{Step 1.} Since $\mathcal X \hookrightarrow C(\overline{\Omega};\mathbb R^N)$, it is clear that $\Psi$ is continuous on $\mathcal X$.  For initial data $y_0\in \mathcal X$, the corresponding maximal solution $y$ of  \eqref{MS} is classical and has strictly positive components. We may thus differentiate $\psi(y)$ with respect to $t \in (0,t^+(y_0))$ and use \eqref{MS} to the result
\begin{equation*}
\rho \partial_t \psi(y) = \sum_{k=1}^N\mu_k \rho\partial_t y = -
\sum_{k=1}^N \text{div}_x (\mu_k J_k) + \sum_{k=1}^N (\nabla_x
\mu_k| J_k) + (\mu| M r(y)). 
\end{equation*} 
Here the first summand vanishes after integration over $\Omega$ due to the boundary conditions $(\nu|J_k)= 0.$ Therefore
\begin{equation}\label{m688}
\rho \partial_t \Psi(y) = \int_{\Omega} \sum_{k=1}^N (\nabla_x
\mu_k(y)| J_k)\, dx + \int_{\Omega} (\mu(y)| M r(y))\, dx.
\end{equation} 
We prove that the \emph{integrands} on the right-hand side in \eqref{m688} are negative. For the second integrand, this is a consequence of Lemma \ref{lem:manifold}. For the first integrand in \eqref{m688} we write
$$\sum_{k=1}^N \scalar{\nabla_x
\mu_k(y)}{J_k} = \sum_{\alpha=1}^n \scalar{\partial_{x_\alpha}
\mu(y)}{J^\alpha}.$$
For fixed $\alpha$ we calculate, using $P(y) J^\alpha = J^\alpha$,
$YP(y)^{\textsf{T}} = P(y) Y$ and that $P(y)Y \partial_{x_\alpha}
\mu  = B(y) J^\alpha$ by \eqref{MSLaw},
\begin{align*}
(\partial_{x_\alpha} \mu(y)|J^\alpha) &  = (P(y) Y
\partial_{x_\alpha} \mu(y)|Y^{-1} J^\alpha) = (B_S(y) Y^{-1/2}
J^\alpha|Y^{-1/2}J^\alpha).
\end{align*}
Here $B_S(y)$ is the symmetrization of $B(y)$ introduced in Section \ref{sec:MS}. Since $B_S(y)$ is negative definite on $Y^{-1/2} \EE$, it follows that $(\partial_{x_\alpha} \mu(y)|J^\alpha)\leq 0$, and that this term vanishes if and only if $J^\alpha = 0$. Hence $\Psi$ decreases along solutions of \eqref{MS}.

\emph{Step 2.} Assume $\Psi$ is not strictly decreasing along a solution $y$. Then there is $t_* \in (0,t^+(y_0))$ such that $\partial_t\Psi(y(t_*)) = 0$. Since both integrands in \eqref{m688} are nonnegative, we obtain that 
\begin{equation}\label{m692}
 \sum_{k=1}^N (\nabla_x
\mu_k(y(t_*))| J_k) = 0, \qquad   (\mu(y(t_*))| M r(y(t_*))) = 0.
\end{equation}
The first identity and the considerations in Step 1 show that $(\partial_{x_\alpha} \mu(y(t_*))|J^\alpha)= 0$, and therefore $J^\alpha = 0$ for each $\alpha = 1,\ldots,n$. Hence $\nabla_x y(t_*) = 0$ by \eqref{MSLaw} and $y(t_*)$ is spatially homogeneous. The second identity in \eqref{m692} and Lemma \ref{lem:manifold} imply $y(t_*)\in \mathcal E$. Thus $y$ is a kinetic equilibrium. This proves that $\Psi$ is a strict Lyapunov function for \eqref{MS}.

\emph{Step 3.} To show b), we note that for any equilibrium $y_*$ of \eqref{MS} we have $\partial_t \Psi(y_*) = 0$. Thus \eqref{m688} and the same arguments as in the previous step show that $y_*$ is  homogeneous and  $y_* \in \mathcal E$. 
\end{proof}

We prove stability with asymptotic phase for the equilibria of \eqref{MS}.

\begin{The}\label{Sat:Haupt2} Assume \emph{(\textbf{R})}. Then any equilibrium $y_*\in \mathcal E$ of \eqref{MS} is stable. Moreover, for each $y_*\in \mathcal E$ there is $\varepsilon > 0$ such that
if 
$$\|y_0 -y_*\|_{W_p^{2-2/p}(\Omega;\mathbb R^N)} \leq \e$$ 
for some $y_0\in \mathcal X$, then the solution of \eqref{MS} corresponding to $y_0$ exists globally in time and  converges at an exponential rate to some $y_\infty\in \mathcal E$, with respect to the $W_p^{2-2/p}(\Omega;\mathbb R^N)$-norm.
\end{The}
\begin{proof} Fix $y_* \in \mathcal E$. To prove the assertions for $y_*$ we intend to apply Theorem \ref{The:Langzeitverhalten1}. Recall from the proof of Theorem \ref{Sat:Haupt1} that \eqref{MS} may be rewritten into the form
\begin{equation}\label{abs-stab}
 \dot{u} + \mathcal A(u)u = F(u), \quad t\in (0,T), \qquad u(0) = u_0,
\end{equation}
where $\mathcal A$ is defined in \eqref{eq:opmathcalA} and $F(u) = M r(u+\frac{1}{N}\textsf{e}).$ The solutions of \eqref{MS} and \eqref{abs-stab} are in one-to-one correspondence via $u = y- \frac{1}{N}\textsf{e}$. It thus suffices to prove the corresponding stability assertions for the equilibrium $u_* = y_*-  \frac{1}{N}\textsf{e}$ of \eqref{abs-stab}. To this end we verify the conditions (i)-(iv) from Theorem \ref{The:Langzeitverhalten1} concerning the linearized operator $u\mapsto\mathcal A_*u = \mathcal A(u_*)u + [\mathcal A'(u_*)u]u_* - F'(u_*)u$ with domain $X_1= \{u\in W_p^2(\Omega;\EE)\,:\,\partial_\nu u = 0\}$. Since $\nabla u_* = 0$, from \eqref{eq:opmathcalA} we see that $\mathcal A(u_*) =- A_0(y_*)\Delta$, and Lemma \ref{Pro:Diffbarkeit} implies $[\mathcal A'(u_*)h]u_*= 0$. Therefore
$$\mathcal A_* = -A_0(y_*)\Delta - Mr'(y_*).$$
In Step 2 of the proof of Lemma \ref{lem:manifold} it was shown that
$$r'(y_*) = \bnu \nabla_y \mathbf r(\rho M^{-1}y_*) = - \bnu K \bnu^{\textsf{T}} Y_*^{-1},$$
where $K = \text{diag}(k_l^-c_*^{\bnu_{l}^-})$. The key observation is the
following identity: for an eigenvalue $\lambda$ of $\mathcal A_*$
with eigenfunction $u \in X_1$ we have
\begin{align}
\lambda \int_\Omega &(u|Y_*^{-1} M^{-1} u) \,dx  =
\int_\Omega (\mathcal A_* u|Y_*^{-1} M^{-1} u)\,dx \notag \\
& = \sum_{\alpha = 1}^n\int_\Omega (A(y_*)P(y_*) M^{-1}
\partial_{x_\alpha}^2 u |Y_*^{-1} M^{-1} u)\,dx + \int_\Omega (M\bnu K\bnu^{\textsf{T}}Y_*^{-1} u|Y_*^{-1} M^{-1}
u)\,dx\notag \\
& = -\sum_{\alpha = 1}^n\int_\Omega (A(y_*)P(y_*) Y_* Y_*^{-1}
M^{-1}
\partial_{x_\alpha} u |Y_*^{-1} M^{-1} \partial_{x_\alpha} u)\,dx
 + \int_\Omega (K\bnu^{\textsf{T}}Y_*^{-1} u|
\bnu^{\textsf{T}}Y_*^{-1} u)\,dx.\label{eq:spec-key}
\end{align}
We now verify the conditions (i)-(iv) from Theorem \ref{The:Langzeitverhalten1}.

(i)+(ii) By the Lemma \ref{lem:manifold} and Proposition \ref{Lem:Konstanten}, the set of equilibria of \eqref{abs-stab} in $\mathcal X$ is given by $\mathcal E - \frac{1}{N}\textsf{e}$ and forms a smooth manifold. The tangent space at $u_*$ is given by $\N(\bnu^{\textsf{T}}Y_*^{-1})\cap \mathds E$. We show that $\N(\mathcal A_*) = \N(\bnu^{\textsf{T}}Y_*^{-1})\cap \mathds E$. Let $u\in \N(\mathcal A_*)$. Since $-A(y_*)P(y_*) Y_*$ is positive semi-definite by Lemma \ref{Lem:4}
 and $K$ is diagonal with positive entries, \eqref{eq:spec-key} with $\lambda =0$ 
yields that
$$\sum_{\alpha = 1}^n\int_\Omega (A(y_*)P(y_*) Y_*
Y_*^{-1} M^{-1}
\partial_{x_\alpha} u |Y_*^{-1} M^{-1} \partial_{x_\alpha} u)\,dx
  = \int_\Omega (K\bnu^{\textsf{T}}Y_*^{-1} u|
\bnu^{\textsf{T}}Y_*^{-1} u)\,dx = 0.$$ Observe that $Y_*^{-1}
M^{-1}
\partial_{x_\alpha} u(x)\notin \N(A(y_*)P(y_*) Y_*)= \text{span}\{\textsf{e}\}$ because
of $\partial_{x_\alpha} u(x)\in \EE$ for all $x\in \Omega$.
Thus the definiteness of $A(y_*)P(y_*) Y_*|_{\mathds
E}$ yields $Y_*^{-1} M^{-1} \partial_{x_\alpha}u = 0$ for all
$\alpha$, hence $u$ is constant. Moreover, the second identity
implies that $\bnu^{\textsf{T}} Y_*^{-1} u = 0$. This shows $\N(\mathcal A_*) \subseteq \N(\bnu^{\textsf{T}}Y_*^{-1})\cap \EE$. The other inclusion is obvious.

(iii) We show that zero is a semi-simple eigenvalue of $\mathcal
A_*$. Since $\mathcal A_*$ has compact resolvent, there is a smallest number $k\in \mathbb N$, the so-called Riesz index, such that $\N(\mathcal A_*^k) = \N(\mathcal A_*^{k+1})$. For the Riesz index $k$ we further have $X_0 = \N(\mathcal A_*^k) \oplus \R(\mathcal A_*^k).$ In fact, this is well-known for compact operators (see \cite[Theorem 3.3]{Kress}) and carries over to $\mathcal A_*$ in a standard way by a resolvent.  We thus have to show that $\N(\mathcal A_*^2) \subseteq \N(\mathcal A_*).$ For $u\in \N(\mathcal A_*^2)$ the identity $\N(\mathcal A_*) =
 \N(\bnu^{\textsf{T}}Y_*^{-1}) \cap \EE$ as shown above implies that $v =
\mathcal A_* u$ is constant and satisfies $\bnu^{\textsf{T}}Y_*^{-1} v = 0$. We therefore have
\begin{align*}
(v|M^{-1} Y_*^{-1} v) = - (M\bnu K\bnu^{\textsf{T}}Y_*^{-1}u|
M^{-1} Y_*^{-1} v) = - ( K\bnu^{\textsf{T}}Y_*^{-1}u|
\bnu^{\textsf{T}} Y_*^{-1} v)= 0,
\end{align*}
which implies $v =0$ and thus $u\in \N(\mathcal A_*)$.

(iv) We finally show that $\sigma(\mathcal A_0)\setminus \{0\}$ is
strictly contained in $\{\text{Re}\,z>0\}$. Because $\mathcal
A_*$ has compact resolvent, its spectrum consists of discrete
eigenvalues, and thus it suffices to show that each eigenvalue
$\lambda \neq 0$ is positive. But  this is a consequence of
\eqref{eq:spec-key} and the positive (semi-) definiteness of
$-A(y_*)P(y_*) Y_*$ and $K$.
\end{proof}

We end the paper with a conditional result on the convergence of
solutions to equilibria.
\begin{Pro} \label{prop:global} Assume \emph{(\textbf{R})} and $y_0\in \mathcal X$. Let $y$ be the solution of \eqref{MS} corresponding to $y_0$. Then the following holds true.
\begin{itemize}
 \item[a)] If $\sup_{t \in (0,t^+(u_0))} \|y(t,\cdot)\|_{W_p^{2-2/p}(\Omega;\mathbb R^N)} <\infty$, then $y$ is a global solution.
 \item[b)] Suppose additionally that, for some $t_0 > 0$, $$\inf_{t > t_0, \,x\in \overline{\Omega}} y(t,x) > 0.$$  Then, as $t\to \infty$,  $y$ converges exponentially with respect to the $W_p^{2-2/p}(\Omega;\mathbb R^N)$-norm to a 
 constant kinetic equilibrium $y_*\in \mathcal E$ of \eqref{MS}. If $r\equiv 0$, then $y_* = \frac{1}{|\Omega|} \int_\Omega y_0\,dx$.
\end{itemize}
\end{Pro}
\begin{proof} Part a) follows from Theorem \ref{The:Langzeitverhalten2}. For Part b), Proposition  \ref{Lem:Konstanten} shows that $\Phi$ is a strict Lyapunov function on $\mathcal X$, and the assertions are a consequence of the Theorems \ref{The:Langzeitverhalten2}  and \ref{Sat:Haupt2}. In case $r\equiv 0$, conservation of mass for each component yields immediately $y_* = \frac{1}{|\Omega|} \int_\Omega y_0\,dx$.\end{proof}

\begin{appendix}
 \section{Abstract quasilinear parabolic evolution equations}\label{sec:abstract}

In this section we will provide some results from \cite{KohPruWil, Bari, PruSimZac} concerning  abstract quasilinear parabolic evolution equations of the form
\begin{equation}\label{eq:Abstrakt}
\dot u + \mathcal{A}(u)u=F(\lambda,u),\quad  t\in (0,T),\qquad u(0)=u_0,
\end{equation}
where $T\in (0,\infty)$ and $\lambda\in\Lambda$ with $\Lambda\subset\mathbb{R}$ open and nonempty. In the sequel let $\mu\in (1/p,1]$ and define a weighted $L_p$-space by
$$L_{p,\mu}(0,T;X)=\{u\colon[0,T]\to X\,:\,[t\mapsto t^{1-\mu}u(t)]\in L_p(0,T;X)\},$$
where $X$ is a Banach space. Furthermore, let
$$W_{p,\mu}^1(0,T;X)=\{u\in L_{p,\mu}(0,T;X)\cap W_{1,loc}^{1}(0,T;X)\,:\,\dot{u}\in L_{p,\mu}(0,T;X)\}.$$
Observe that for $\mu=1$ we have $L_{p,1}=L_p$ as well as $W_{p,1}^1=W_p^1$.

Consider two Banach spaces $X_1$ and $X_0$ with dense embedding $X_1\hookrightarrow X_0$. We are interested in solutions $u$ of \eqref{eq:Abstrakt} with regularity
$$u\in W_{p,\mu}^1(0,T;X_0)\cap L_{p,\mu}(0,T;X_1)=:\mathbb{E}_{1,\mu}(0,T).$$
By \cite[Proposition 3.1]{PruSim} the embedding
$$\mathbb{E}_{1,\mu}(0,T)\hookrightarrow BUC(0,T;(X_0,X_1)_{\mu-1/p,p})$$
holds true. This yields $u_0\in (X_0,X_1)_{\mu-1/p,p}=:X_{\gamma,\mu}$ as a necessary condition for the initial value $u_0$ from \eqref{eq:Abstrakt}. It is possible to show that this regularity assumption on $u_0$ is also sufficient for solving \eqref{eq:Abstrakt} in the space $\mathbb{E}_{1,\mu}(0,T)$. To formulate the precise statement we introduce some notation. By $\mathcal{B}(X_1,X_0)$ we denote the space of all bounded and linear operators from $X_1$ to $X_0$. Furthermore, we say that an operator $\mathcal{A}_0\colon X_1\to X_0$ has maximal regularity of type $L_p$ if for each $f\in L_p(0,T_0;X_0)=:\mathbb{E}_{0,\mu}(0,T_0)$ there exists a unique solution
$$u\in W_p^1(0,T_0;X_0)\cap L_p(0,T_0;X_1)$$
of the linear problem
$$\dot{u}+\mathcal{A}_0 u=f,\ t\in [0,T_0],\quad u(0)=0,$$
where $T_0>0$ is arbitrary.

Concerning well-posedness of \eqref{eq:Abstrakt} and regularity of the solution of \eqref{eq:Abstrakt} we have the following result. 
\begin{The}\label{The:lokaleExistenz}
Let $p\in (1,\infty)$, $\mu \in (1/p,1]$, $\emptyset\neq V_\mu \subseteq X_{\gamma,\mu}$ an open set, $\mathcal{A}\in C^{1}(V_\mu; \B(X_1,X_0))$ and $F\in C^1(\Lambda\times V_\mu;X_0)$. Assume that $\mathcal{A}(u)$ has maximal regularity of type $L_p$ for each $u\in V_\mu$.
Then the following statements are true.
\begin{itemize}
\item[a)] For each $(u_0,\lambda_0)\in V_\mu\times\Lambda$ there exists $T>0$ such that the problem \eqref{eq:Abstrakt}
has a unique solution $u(\cdot,u_0,\lambda_0)\in\mathbb{E}_{1,\mu}(0,T)\cap BUC(0,T;V_\mu)$.
\item[b)] Each local solution of \eqref{eq:Abstrakt} can be extended to a maximal solution with a maximal interval of existence $J(u_0,\lambda_0)=[0,t^+(u_0,\lambda_0))$ and $u(\cdot,u_0,\lambda_0)\in\mathbb{E}_{1,\mu}(0,T)\cap BUC(0,T;V_\mu)$ for each $T\in (0,t^+(u_0,\lambda_0))$.
\item[c)] For any $T\in (0,t^+(u_0,\lambda_0))$ there exists $\delta>0$ such that $t^+(u_1,\lambda_1)>T$ for all $(u_1,\lambda_1)\in B_{\delta}^{X_{\gamma,\mu}}(u_0)\times B_{\delta}(\lambda_0)$ and the mapping
    $$B_{\delta}^{X_{\gamma,\mu}}(u_0)\times B_{\delta}(\lambda_0)\to C([0,T];X_{\gamma,\mu}),\quad (u_1,\lambda_1)\mapsto u(\cdot,u_1,\lambda_1)$$
    is continuously Fr\'{e}chet differentiable. In particular, the mapping $(u_1,\lambda_1)\mapsto t^+(u_1,\lambda_1)$ is lower semicontinuous.
\item[d)] For each  $T\in (0,t^+(u_0,\lambda))$ we have
$$u(\cdot,u_0,\lambda)\in C^{1}((0,T];X_\gamma) \cap C^{2 - 1/p}((0,T];X_0) \cap C^{1 - 1/p}((0,T];X_1),$$
where $X_\gamma=X_{\gamma,1}=(X_0,X_1)_{1-1/p,p}$.
\end{itemize}
\end{The}
\begin{proof}
a) Let $(u_0,\lambda_0)\in V_\mu\times\Lambda$ be fixed. Then it follows from the assumption on $F$ that $F_{\lambda_0}:=F(\lambda_0,\cdot)$ is locally Lipschitz continuous in $V_\mu$. Therefore the statement follows from \cite[Theorem 2.1]{KohPruWil}.

b) This is a direct consequence of \cite[Corollary 2.2]{KohPruWil}.

c) For the proof of this assertion we invoke the implicit function theorem. Let $T\in (0,t^+(u_0,\lambda_0))$ be given and let $\iota:\mathbb{E}_{1,\mu}(0,T)\to BUC(0,T;X_{\gamma,\mu})$ denote the inclusion map
$$\mathbb{E}_{1,\mu}(0,T)\hookrightarrow BUC(0,T;X_{\gamma,\mu}).$$
Since $V_\mu$ is open in $X_{\gamma,\mu}$ is follows that $W_\mu:=\iota^{-1}(BUC(0,T;V_\mu))$ is open in $\mathbb{E}_{1,\mu}(0,T)$. Define a mapping $H:\Lambda\times V_\mu\times W_\mu\to \mathbb{E}_{0,\mu}(0,T)\times X_{\gamma,\mu}$ by
$$H(\lambda_1,u_1,u):=(\dot{u}+\mathcal{A}(u)u-F(\lambda_1,u),u|_{t=0}-u_1).$$
Note that we use the same notation for $(\mathcal{A},F)$ and the corresponding substitution operators induced by $(\mathcal{A},F)$ in $W_\mu$.

From the assumptions on $(\mathcal{A},F)$ it follows readily that $H\in C^1(\Lambda\times V_\mu\times W_\mu;\mathbb{E}_{0,\mu}(0,T)\times X_{\gamma,\mu})$ and the derivative of $H$ with respect to $u$ is given by
$$D_uH(\lambda_1,u_1,u)v=(\dot{v}+\mathcal{A}(u)v+[\mathcal{A}'(u)v]u-D_uF(\lambda_1,u)v,v|_{t=0}),$$
where $v\in \mathbb{E}_{1,\mu}(0,T)$. Observe that $H(\lambda_0,u_0,u_0^*)=0$, where $u_0^*:=u(\cdot,u_0,\lambda_0)$ is the solution of \eqref{eq:Abstrakt} on $[0,T]$ with $\lambda=\lambda_0$. If we want to solve the equation $H(\lambda_1,u_1,u)=0$ for $u$ in a neighborhood of $(\lambda_0,u_0,u_0^*)$, we have to show that the mapping
$$D_uH(\lambda_0,u_0,u_0^*):\mathbb{E}_{1,\mu}(0,T)\to \mathbb{E}_{0,\mu}(0,T)\times X_{\gamma,\mu}$$
is an isomorphism. For $v\in\mathbb{E}_{1,\mu}(0,T)$ we define
$$\mathcal{A}(t)v:=\mathcal{A}(u_0^*(t))v\quad\text{and}\quad B(t)v:=D_uF(\lambda_0,u_0^*(t))v-[\mathcal{A}'(u_0^*(t))v]u_0^*(t).$$
It follows that $\mathcal{A}(\cdot)\in C([0,T];\mathcal{B}(X_1,X_0))$, $B(\cdot)\in L_{p,\mu}(0,T;\mathcal{B}(X_{\gamma,\mu},X_0))$ and for each $t\in [0,T]$ the operator $\mathcal{A}(t)$ has maximal regularity of type $L_p$. We have to show that for each $v_0\in X_{\gamma,\mu}$ and $f\in L_{p,\mu}(0,T;X_0)$ the nonautonomous problem
\begin{equation}\label{eq:nonautonom}
\dot{v}+\mathcal{A}(t)v=B(t)v+f,\ t\in[0,T],\quad v(0)=v_0,
\end{equation}
has a unique solution $v\in \mathbb{E}_{1,\mu}(0,T)$. Note that without loss of generality we may assume that $v_0=0$ by the expense that $f$ has to be replaced by some modified function $\tilde{f}\in L_{p,\mu}(0,T;X_0)$. The strategy for solving \eqref{eq:nonautonom} with $v_0=0$ is as follows. In a first step one applies a Neumann series argument in the space 
$$_0\mathbb{E}_{1,\mu}(0,T):=\{v\in\mathbb{E}_{1,\mu}(0,T):v(0)=0\}$$ 
to obtain a unique solution $v_1\in\!_0\mathbb{E}_{1,\mu}(0,T)$ of \eqref{eq:nonautonom} on some interval $[0,\tau]\subset[0,T]$. Note that $v_1(\tau)\in X_\gamma$, since 
$$\mathbb{E}_{1,\mu}(0,\tau)\hookrightarrow \mathbb{E}_{1,1}(\delta,\tau)\hookrightarrow C([\delta,\tau];X_\gamma).$$
In a second step one solves \eqref{eq:nonautonom} on $[\tau,T]$ with $v_0$ replaced by $v_1(\tau)\in X_\gamma$. Since $B(\cdot)\in L_p(\tau,T;\mathcal{B}(X_\gamma,X_0))$ we may apply \cite[Corollary 3.4]{Bari} to conclude the existence of a unique $v_2\in \mathbb{E}_{1,1}(\tau,T)$ solving \eqref{eq:nonautonom} on $[\tau,T]$ with $v_2(\tau)=v_1(\tau)$.

The implicit function theorem yields $\delta>0$ and a mapping $\Phi\in C^1(B_\delta(\lambda_0)\times B_\delta^{X_{\gamma,\mu}}(u_0);W_\mu)$ such that $H(\lambda_1,u_1,\Phi(\lambda_1,u_1))=0$ for all $(\lambda_1,u_1)\in B_\delta(\lambda_0)\times B_\delta^{X_{\gamma,\mu}}(u_0)$. Therefore  $u_1^*:=\Phi(\lambda_1,u_1)\in W_\mu$ is the unique solution of \eqref{eq:Abstrakt} on $[0,T]$ with $\lambda=\lambda_1$ and $u_0=u_1$. This proves the third assertion, since $\mathbb{E}_{1,\mu}(0,T)\hookrightarrow C([0,T];X_{\gamma,\mu})$.

d) Note that for each $\delta\in (0,T)$ the embedding
$$\mathbb{E}_{1,\mu}(0,T)\hookrightarrow \mathbb{E}_{1,1}(\delta,T)$$
is valid, which shows instant smoothing of the solution. In particular, for each $t_*\in (0,t^+(u_0))$ one has $u(t_*,u_0)\in V_\mu\cap X_\gamma$, since
$$\mathbb{E}_{1,1}(\delta,T)\hookrightarrow BUC(\delta,T;X_\gamma).$$
Moreover the regularity condition on $(A,F_\lambda)$ with $F_\lambda=F(\lambda,\cdot)$ implies that
$$(\mathcal{A},F_\lambda)\in C^{1}(V_\mu\cap X_\gamma; \B(X_1,X_0) \times X_0),$$
and $V_\mu\cap X_\gamma$ is open in $X_\gamma$, since $X_\gamma\hookrightarrow X_{\gamma,\mu}$. We are now in a position to apply \cite[Theorem 5.1]{Bari} to prove the last assertion. Actually in \cite{Bari} the author uses the assumption $(\mathcal{A},F_\lambda)\in C^1(X_\gamma,\mathcal{B}(X_1,X_0)\times X_0)$, i.e.\ $V_\mu\cap X_\gamma=X_{\gamma}$. However, an inspection of the proof shows that the statement remains true if one replaces $X_\gamma$ by the open set $X_\gamma\cap V_\mu\subset X_\gamma$.
\end{proof}
From now on we fix $\lambda\in\Lambda$ and simply write $F$ instead of $F_\lambda=F(\lambda,\cdot)$. We are interested in the qualitative behaviour of a solution $u(\cdot,u_0)$ of \eqref{eq:Abstrakt} if the initial value $u_0$ is close to an equilibrium. We call $u_*$ an equilibrium of \eqref{eq:Abstrakt} if $u_*\in V_\mu\cap X_1$ and $\mathcal{A}(u_*)u_*=f(u_*)$. The set of all equilibria is denoted by $\mathcal{E}$.

Let us recall that the full linearization  of \eqref{eq:Abstrakt} at an equilibrium $u_*\in V_\mu\cap X_1$ is given by
$$\mathcal{A}_0=\mathcal{A}(u_*)+\frac{d}{dw}[\mathcal{A}(w)u_*]|_{w=u_*}-F'(u_*).$$
Then we can state the following result.
\begin{The}\label{The:Langzeitverhalten1}
Let $p\in (1,\infty)$ and $V_\mu\subseteq X_{\gamma,\mu}$ an open subset. Further let $u_* \in V_\mu\cap X_1$ be an equilibrium of \eqref{eq:Abstrakt}.
Assume that $\mathcal{A}(u_*)$ has maximal regularity of type $L_p$ and $(\mathcal{A},F)\in C^1(V_\mu;\B(X_1,X_0)\times X_0)$. Let $\mathcal{A}_0$ denote the full linearization of \eqref{eq:Abstrakt} at $u_*$. Suppose that the equilibrium is \emph{normally stable}, that is
\begin{enumerate}
\item[(i)]
in a neighborhood of $u_*$ the set of equilibria $\mathcal E \subseteq V_\mu\cap X_1$ is a $C^1$-manifold of dimension $m\in \mathbb{N}_0$;
\item[(ii)] the tangent space on $\mathcal E$ at $u_*$ is given by $\N(\mathcal{A}_0)$;
\item[(iii)] $0$ is a semisimple eigenvalue of $\mathcal{A}_0$, i.e.\ $\N(\mathcal{A}_0) \oplus \R(\mathcal{A}_0)=X_0$;
\item[(iv)] the spectrum $\sigma(\mathcal{A}_0)$ satisfies $\sigma(\mathcal{A}_0)\setminus\{0\}\subset \mathbb \{\operatorname{Re}\,z>0\}$.
\end{enumerate}
Then $u_*$ is stable in $X_\gamma$ and there exists $\delta >0$ such that for each $u_0\in B_\delta^{X_\gamma}(u_*)$ the unique solution $u(\cdot,u_0)$ of $\eqref{eq:Abstrakt}$ exists for all $t\ge 0$ and it converges at an exponential rate to some equilibrium $u_\infty \in \mathcal E$ as $t\to \infty$.
\end{The}
\begin{proof}
If $V_\mu\subset X_{\gamma,\mu}$ is open, then $V_\mu\cap X_\gamma$ is open in $X_\gamma$. Furthermore, the regularity assumption on $(A,F)$ implies that
$$(\mathcal{A},F)\in C^1(V_\mu\cap X_\gamma;\mathcal{B}(X_1,X_0)\times X_0).$$
Now the assertion follows directly from \cite[Theorem 2.1]{PruSimZac}.
\end{proof}
It is possible to extend this local result on the qualitative behavior to a global one if one assumes the existence of a strict Lyapunov function for \eqref{eq:Abstrakt}. This assertion is a part of the following result.
\begin{The} \label{The:Langzeitverhalten2}
Let $p\in (1,\infty)$, $\mu \in (1/p,1)$, $V_\mu \subseteq X_{\gamma, \mu}$ an open set. Assume that $\mathcal{A}(u)$ has maximal regularity of type $L_p$ for each $u\in V_\mu$. Further let $(\mathcal{A},F) \in C^1(V_\mu; \B(X_1,X_0) \times X_0)$ and suppose that the embedding  $X_\gamma \hookrightarrow X_{\gamma,\mu}$ is compact.
Suppose that the unique solution $u(\cdot,u_0)$ of \eqref{eq:Abstrakt} with initial value $u_0\in V_\mu$ satisfies
$$u\in BC([\tau,t^+(u_0)); V_{\mu}\cap X_\gamma)$$
for some $\tau\in (0,t^+(u_0))$ as well as
$$\dist(u(t,u_0), \partial V_\mu)\ge\eta>0$$
for all $t\in [0,t^+(u_0))$. Then the following statements are valid.
\begin{enumerate}
\item The solution $u(t,u_0)$ exists for all $t\ge 0$, i.e.\ $t^+(u_0)=\infty$ and the set $\{u(t,u_0) \,:\, t\geq \tau\}$ is relatively compact in $X_\gamma$. Furthermore the omega limit set
$$\omega(u_0) = \{ v\in V_\mu\cap X_\gamma\,:\,\exists\ t_n \to \infty\ \text{s.t.}\ \lim_{n\to\infty}u(t_n,u_0)=v\ \text{in}\ X_\gamma\}$$
is nonempty, compact, connected and positively invariant.
\item If there exists a strict Lyapunov function $\Phi\in C(V_\mu\cap X_\gamma)$ for \eqref{eq:Abstrakt}, then $\omega(u_0)\subset \mathcal{E}$. Moreover, if in addition there exists an equilibrium $u_*\in\omega(u_0)$ which is normally stable, then
$$\lim_{t\to\infty} u(t,u_0)=u_*$$
in $X_\gamma$ at an exponential rate.
\end{enumerate}
\end{The}
\begin{proof}
The assertions follow from \cite[Theorem 3.1 \& Theorem 4.3]{KohPruWil}.
\end{proof}
\end{appendix}

\end{document}